\documentclass[10pt,reqno]{amsart}
\usepackage{amsmath} 
\usepackage{amssymb}
\usepackage{dsfont}
\usepackage[dvips,draft,final]{graphics}
\usepackage{amssymb,amsmath}
\usepackage[T1]{fontenc}
\usepackage{lmodern}
\usepackage{fancyhdr}
\usepackage{url}

\usepackage{color}
\usepackage{graphicx}

\newtheorem{theorem}{Theorem}[section]
\newtheorem{proposition}[theorem]{Proposition}
\newtheorem{lemma}[theorem]{Lemma}

\newtheorem{remark}[theorem]{Remark}
\theoremstyle{definition}

\newcommand{\bel}{\begin{equation} \label}
\newcommand{\ee}{\end{equation}}

\newcommand{\C}{{\mathbb C}}
\newcommand{\R}{{\mathbb R}}
\newcommand{\N}{{\mathbb N}}

\newcommand{\im}{\mathfrak I}

\def\epsilon{\varepsilon}
\def\phi {\varphi}
\def\beq{\begin{equation}}
\def\eeq{\end{equation}}
\newcommand{\bea}{\begin{eqnarray}}
\newcommand{\eea}{\end{eqnarray}}
\newcommand{\beas}{\begin{eqnarray*}}
\newcommand{\eeas}{\end{eqnarray*}}

{
\renewcommand{\leq}{\leqslant}
\renewcommand{\geq}{\geqslant}
\providecommand{\abs}[1]{\left\lvert#1\right\rvert}
\providecommand{\norm}[1]{\left\lVert#1\right\rVert}

\begin{document}

\title[Borg-Levinson theorem for magnetic Schr\"odinger operators]{A multidimensional Borg-Levinson theorem for magnetic Schr\"odinger operators with partial spectral data }

\author[Yavar Kian]{Yavar Kian}
\address{Aix Marseille Univ, Universit\'e de Toulon, CNRS, CPT, Marseille, France.}
\email{yavar.kian@univ-amu.fr}
\maketitle

\begin{abstract} We  consider  the multidimensional Borg-Levinson theorem of determining both the magnetic field  $dA$ and the electric potential $V$, appearing in the Dirichlet  realization of the magnetic  Schr\" odinger operator $H=(-{\rm i}\nabla+A)^2+V$ on a bounded domain $\Omega\subset\R^n$, $n\geq2$, from partial knowledge of the boundary spectral data of $H$. The full boundary spectral data are given by  the set  $\{(\lambda_{k},{\partial_\nu \phi_{k}}_{|\partial\Omega}):\ k\geq1\}$, where $\{ \lambda_k:\ k\in \N^* \}$ is the non-decreasing sequence of eigenvalues of $H$,  $\{ \phi_k:\ k\in \N^* \}$ an associated Hilbertian basis of eigenfunctions and $\nu$ is the unit outward normal vector to $\partial\Omega$. We prove that some asymptotic knowledge of  $(\lambda_{k},{\partial_\nu \phi_{k}}_{|\partial\Omega})$ with respect to $k\geq1$ determines uniquely the magnetic field  $dA$ and the electric potential $V$.
\end{abstract}
\vspace{1cm}

\noindent{\bf Keywords: }{
Inverse spectral problem, Borg-Levinson theorem,  magnetic Schr\"odinger operators.\\

}

\section{Introduction}
\subsection{Statement of the problem}
 We consider $\Omega \subset {\Bbb R}^n$, $n\geq2$,   a $\mathcal C^{1,1}$ bounded and  connected domain  such that $\R^n\setminus\Omega$ is also connected. We set $\Gamma=\partial\Omega$. Let $A\in W^{1,\infty}(\Omega, \R^n)$, $V\in L^\infty(\Omega,\R)$ and consider the  magnetic Schr\"odinger operator $H=(-{\rm i}\nabla+A)^2+V$ acting on $L^2(\Omega)$ with domain $D(H)=\{v\in H^1_0(\Omega):\ (-{\rm i}\nabla+A)^2v\in L^2(\Omega)\}$.

Let $A_j\in W^{1,\infty}(\Omega, \R^n)$, $V_j\in L^\infty(\Omega,\R)$,  $j=1,2$, and consider the  magnetic Schr\"odinger operators $H_j=H$ for $A=A_j$ and $V=V_j$,  $j=1,2$. We say that $H_1$ and $H_2$ are  gauge equivalent if there exists $p\in W^{2,\infty}(\Omega,\R)\cap H^1_0(\Omega)$ such that $H_2=e^{-{\rm i}p}H_1e^{{\rm i}p}$.

It is well known that $H$ is a selfadjoint operator. By the compactness of the embedding $H^1_0(\Omega) \hookrightarrow L^2(\Omega)$, the spectrum of $H$ is
purely discrete. We note $\{ \lambda_k:\ k\in \N^* \}$  the non-decreasing sequence of eigenvalues of $H$ and  $\{ \phi_k:\ k\in \N^* \}$ an associated Hilbertian basis of eigenfunctions. In the present paper we consider the Borg-Levinson inverse spectral problem of determining uniquely $H$,   modulo gauge equivalence, from partial knowledge of the boundary spectral data $\{ (\lambda_k,{\partial_\nu\phi_k}_{|\Gamma}):\ k\in \N^* \}$ with $\nu$ the outward unit normal vector to $\Gamma$. Namely, we prove that some asymptotic knowledge of $(\lambda_k,{\partial_\nu\phi_k}_{|\Gamma})$ with respect to $k\in \N^* $ determines uniquely the operator $H$ modulo gauge transformation.
\subsection{ Borg-Levinson   inverse spectral problems }

It is Ambartsumian  who first investigated in 1929 the inverse spectral problem of determining the real potential $V$ appearing in the Schr\"odinger operator $H=-\Delta+V$, acting in $L^2(\Omega)$, from partial spectral data of $H$. For $\Omega=(0,1)$, he proved in \cite{A} that $V=0$ if the spectrum of the Neumann realization of $H$ equals $\{ k^2:\ k \in \N \}$. For the same operator, but endowed with homogeneous Dirichlet boundary conditions, Borg \cite{B} and Levinson \cite{L} established that the Dirichlet spectrum $\{ \lambda_k:\ k \in \N^*\}$ does not uniquely determine $V$. They showed that additional spectral data, namely $\{ \| \phi_k \|_{L^2(0,1)}:\ k \in \N^* \}$, where $\{ \varphi_k:\ k \in \N^* \}$ is an $L^2(0,1)$-orthogonal basis of eigenfunctions of $H$ obeying the condition $\varphi_k'(0)=1$, is needed. Gel'fand and Levitan proved in \cite{GL} that uniqueness is still valid upon substituting the terminal velocity $\phi_k'(1)$ for $\| \phi_k \|_{L^2(0,1)}$ in the one-dimensional Borg and Levinson theorem.

In 1988, Nachman, Sylvester, Uhlmann \cite{NSU} and  Novikov \cite{No} proposed a multidimensional formulation of the result of Borg  and Levinson. Namely, they proved that the  boundary spectral data $\{ (\lambda_k ,{\partial_\nu \phi_k}_{\vert\partial\Omega}):\ k \in \N^* \}$, where $\nu$ denotes the outward unit normal vector to $\partial \Omega$ and $(\lambda_k, \phi_k)$ is the $k^{\rm th}$ eigenpair of $-\Delta+V$, determines uniquely  the Dirichlet realization of the operator $-\Delta+V$.  The initial formulation of the multidimensional Borg-Levinson theorem by \cite{NSU} and \cite{No}  has been improved in several ways by various authors. Isozaki \cite{I} (see also \cite{Ch}) extended the result of \cite{NSU} when finitely many eigenpairs remain unknown, and, recently, Choulli and Stefanov \cite{CS} claimed uniqueness in the determination of $V$ from the asymptotic behavior of $(\lambda_k,{\partial_\nu \phi_k}_{\vert\Gamma})$ with respect to $k$.
Moreover, Canuto and Kavian \cite{ CK1, CK2} considered the determination  of the conductivity $c$, the electric potential $V$ and the weight $\rho$   from  the boundary spectral data of the
operator $\rho^{-1}(-div (c \nabla\cdot) + V)$ acting on the weighted space $L^2_\rho(\Omega)$ endowed with either Dirichlet or Neumann boundary conditions. Namely, \cite{ CK1, CK2} proved that the boundary spectral data of $\rho^{-1}(-div (c \nabla\cdot) + V)$ determines uniquely two of the three coefficients $c$, $V$ and $\rho$. The case of magnetic Schr\"odinger operator has been treated by \cite{Ser} who determined both the magnetic field $dA$ and the electric potential $V$ of the operator  $H=(-{\rm i}\nabla+A)^2+V$. Here the 2-form $dA$ of a vector valued function $A=(a_1,\ldots,a_n)$ is defined by
\[dA=\sum_{i<j}(\partial_{x_j}a_i-\partial_{x_i}a_j)dx_j\wedge dx_i.\]

All the above mentioned results were obtained with $\Omega$ bounded and operators of purely discrete spectral type. In some recent work \cite{KKS} examined a Borg-Levinson inverse problem stated in an  infinite cylindrical waveguide for Schr\"odinger operators with purely absolutely continuous spectral type. More precisely, \cite{KKS} proved that a real potential $V$ which is $2\pi$-periodic along the axis of the waveguide is uniquely determined 
by  some asymptotic knowledge of the boundary  Floquet spectral data of the Schr\"odinger operator $-\Delta+V$ with Dirichlet boundary conditions. 

Finally, let us mention for the sake of completeness that the stability issue in the context of Borg-Levinson inverse problems was examined in \cite{ BCY, BF, Ch, CS, KKS} and
that \cite{BK,BF, KK} established related results  on Riemannian manifolds. We also precise that \cite{Mo,SU1,SU2} have proved  stability estimates in the recovery of  coefficients from the hyperbolic Dirichlet-to-Neumann map which is equivalent to  the determination of general Schr\"odinger operators from boundary spectral data.

\subsection{Main result}
Let $A_j\in W^{1,\infty}(\Omega, \R^n)$, $V_j\in L^\infty(\Omega,\R)$ and consider the  magnetic Schr\"odinger operators $H_j=H$ for $A=A_j$ and $V=V_j$, $j=1,2$. Further we note $(\lambda_{j,k},\phi_{j,k})$, $k\geq1$, the $k^{\rm th}$ eigenpair of $H_j$, for $j=1,2$. Our main result can be stated as follows.

\begin{theorem}
\label{thm-1} 
 We fix $\Omega_1$ an arbitrary open neighborhood of $\Gamma$ in  $\Omega$ $($$\Gamma\subset\overline{\Omega_1}$ and $\Omega_1\subsetneq\Omega$$)$. For $j = 1, 2$, let $V_j \in L^\infty(\Omega,\R)$ and  let $A_j\in  \mathcal C^1(\overline{\Omega},\R^n)$ fulfill 
\bel{t1a}A_1(x)=A_2(x),\quad x\in\Omega_1.\ee
Assume that the conditions
\begin{equation}\label{tt2a} 
\lim_{k\to+\infty}\abs{\lambda_{1,k}-\lambda_{2,k}}=0,\quad \sum_{k=1}^{+\infty}\norm{\partial_\nu\phi_{1,k}-\partial_\nu\phi_{2,k} }_{L^2(\Gamma)}^2<\infty\end{equation}
hold simultaneously. Then, we have $dA_1=dA_2$ and $V_1 = V_2$.
\end{theorem}
Note that condition \eqref{t1a} corresponds to the knowledge of the magnetic potential on a neighborhood of the boundary.

Let us observe that, as mentioned by \cite{CS,KKS}, Theorem \ref{thm-1} can be considered as a uniqueness theorem under the assumption that
the spectral data are asymptotically "very close". Conditions \eqref{tt2a}  are similar to the one considered by \cite{KKS} and they are weaker than the requirement that 
\[\abs{\lambda_{1,k}-\lambda_{2,k}}\leq Ck^{-\alpha},\quad \norm{\partial_\nu\phi_{1,k}-\partial_\nu\phi_{2,k} }_{L^2(\Gamma)}\leq Ck^{-\beta}\]
for some $\alpha>1$ and $\beta>1-{1\over 2n}$, considered in \cite[Theorem 2.1]{CS}. Note also that conditions \eqref{tt2a} are weaker than the knowledge of the boundary spectral data  with a finite number of data missing considered by \cite{I}.

Let us remark that there is an obstruction to uniqueness given by the gauge invariance of  boundary spectral data for magnetic Shr\"odinger operators. More precisely, let $p\in \mathcal C^\infty_0(\Omega\setminus\overline{\Omega_1})\setminus\{0\}$ and assume that  $A_1=\nabla p +A_2\neq A_2$, $V_1=V_2$. Then, we have $H_1=e^{-{\rm i}p}H_2e^{{\rm i}p}$ and one can check that we can choose the spectral data of $H_1$ and $H_2$ in such a way that the conditions
 \[{\partial_\nu\phi_{1,k}}_{|\Gamma}={\partial_\nu\phi_{2,k}}_{|\Gamma},\quad \lambda_{1,k}=\lambda_{2,k},\quad k\in\mathbb{N}^*\]
are fulfilled. Therefore, conditions \eqref{t1a}-\eqref{tt2a} are fulfilled but $H_1\neq H_2$. Nevertheless,  assuming \eqref{t1a} fulfilled, the conditions $dA_1=dA_2$ and $V_1 = V_2$ imply that $H_1$ and $H_2$ are gauge equivalent. Therefore, Theorem \ref{thm-1} is equivalent to the unique determination of  magnetic Schr\"odinger operators  modulo gauge transformation from the asymptotic  knowledge of the boundary spectral data given by  conditions \eqref{tt2a}.

We stress out that the problem under examination in this text is a  Borg-Levinson inverse problem for the  magnetic Schr\"odinger operator $H=(-{\rm i}\nabla+A)^2+V$.  To our best knowledge, there are only two multi-dimensional Borg-Levinson uniqueness result for magnetic Schr\"odinger operators  available in the mathematical literature, \cite[Theorem B]{KK} and \cite[Theorem 3.2]{Ser} (we refer also to \cite{Ni} for related inverse scattering results). In \cite{KK}, the authors considered  general magnetic Schr\"odinger operators with smooth coefficients on a smooth connected Riemannian manifold and they proved unique determination of these operators modulo gauge invariance from the knowledge of the boundary spectral data with a missing finite number of data. In \cite{Ser}, Serov treated  this problem on a bounded domain of $\R^n$, and he proved that, for $A\in W^{1,\infty}(\Omega,\R^n)$ and $V\in L^\infty(\Omega,\R)$,  the full boundary spectral data $\{ (\lambda_k ,{\partial_\nu \phi_k}_{|\Gamma}):\ k \in \N^* \}$ determines uniquely  $dA$ and $V$. In contrast to \cite{KK, Ser}, in the present paper we prove that the asymptotic knowledge  of the boundary spectral data, given by the conditions \eqref{tt2a},   is sufficient for the unique determination of $dA$ and $V$. To our best knowledge, conditions \eqref{tt2a} are the weakest conditions on boundary spectral data that guaranty uniqueness of magnetic Schr\"odinger operators modulo gauge transformation. Moreover, our uniqueness result is stated with   conditions similar to   \cite[Theorem 1.4]{KKS}, which seems to be the most precise  Borg-Levinson uniqueness result so far for Schr\"odinger operators without magnetic potential ($A=0$).

An important ingredient in our analysis is a suitable representation that allows to express the magnetic potential $A$ and the electric potential $V$ in terms of Dirichlet-to-Neumann map associated to the equations $(-{\rm i}\nabla +A)^2u+Vu-\lambda u=0$ for some $\lambda\in\mathbb C$. In \cite{I} Isozaki applied a similar approach to the Schr\"odinger operator $-\Delta+V$ with Dirichlet boundary condition\footnote{This argument was inspired by the Born approximation method of the scattering theory.} and \cite{CS,KKS} applied the representation formulas of \cite{I}. Inspired by the construction of complex geometric optics solutions of \cite{BC,FKSU,KLU,KU,NSU1,Sa1,Su}  we prove that the approach of  \cite{CS,I,KKS} can be extended to magnetic Schr\"odinger operators. More precisely, we derive two representation formulas that allow to recover both the magnetic field and the electric potential of magnetic Schr\"odinger operators  which means recovery of both coefficients of order one and zero in contrast to \cite{CS,I,KKS} where only determination of coefficients of order zero is considered. This paper is the first where the extension of the approach developed by \cite{I} to more general coefficients than coefficients of order zero is considered. Note also that our approach make it possible to prove this extension  without imposing important assumptions of regularity of the  admissible coefficients.

We believe that the approach developed in the present paper can be used for results of stability in the determination of the magnetic field $dA$ and the electric potential $V$ similar to  \cite[Theorem 1.3]{KKS}. Indeed, following  the strategy set in this paper we expect a stability estimate associated to the the determination of the magnetic field $dA$. The main issue comes from the stability in the determination of the electric potential $V$. Nevertheless, we believe that this problem can be solved by adapting suitably the  argument developed in \cite{T} related to the inversion of the $d$ operator on differential forms restricted to the right subspaces.

\subsection{Outline}
This paper is organized as follows. In Section 2 we consider some useful preliminary results concerning solutions of equations of the form $(-{\rm i}\nabla +A)^2u+Vu-\lambda u=0$ for some $\lambda\in\mathbb C\setminus\sigma(H)$. In Section 3 we introduce two representation formulas making the connection between the Dirichlet-Neumann map associated with the previous equations and the couple  $(A,V)$ of  magnetic and electric potential. Finally, in section 4 we combine all these results and we prove Theorem \ref{thm-1}.

\section{Notations and preliminary results}
\label{sec:Prelim}

In this section we introduce some notations and we give some properties of solution of the equation $(-{\rm i}\nabla +A)^2u+Vu-\lambda u=0$.
We  denote  by $\langle f,\psi\rangle$ the duality between $\psi \in H^{1/2}(\Gamma)$ and $f$ belonging to the dual $H^{-1/2}(\Gamma)$ of $H^{1/2}(\Gamma)$. However, when in $\langle f,\psi \rangle $ both $f$ and $\psi$ belong to $L^2(\Gamma)$, to make things simpler $\langle \cdot,\cdot\rangle$ can be interpreted as the scalar product of $L^2(\Gamma)$, namely
$$\langle f,\psi\rangle = \int_{\Gamma} f(x)\, \overline{\psi(x)}\,d\sigma(x).$$
\medskip
We introduce the operator $H$  defined as
\begin{equation}\label{eq:Def-A-theta} H u := (-{\rm i}\nabla+A)^2 u + V u,\quad u\in D(H) := \left\{\psi \in H^1_0(\Omega) \; ; \;(-{\rm i}\nabla+A)^2\psi \in L^2(\Omega)\right\}.\end{equation}
Recall that $H$ is associated to the quadratic form $b$ given by
$$b(u,v) = \int_{\Omega}(-{\rm i}\nabla+A) u(x)\cdot\overline{(-{\rm i}\nabla+A) v(x)}\,dx + \int_{\Omega}V(x)\,u(x)\overline{v(x)}\,dx,\quad u,v\in H^1_0(\Omega).$$
Moreover,  the spectrum of $H$ is discrete and composed of the non-decreasing sequence of eigenvalues denoted by $\sigma (H) = \{\lambda_{k}\; ; \; k \geq 1\}$. If we write $V = V^+ - V^-$, with $V^\pm := \max(0,\pm V)$, we have that the spectrum $\sigma(H)$ of $H$ is contained into 
$ [-\|V^-\|_{L^\infty(\Omega)}, + \infty)$.
According to  \cite[Theorem 2.2.2.3]{Gr},   we can show that $D(H)$ embedded continuously into $H^2(\Omega)$. Therefore the eigenfunctions  $(\phi_{k})_{k\geq1}$ of $H$, that form an Hilbertian basis, are lying in $H^2(\Omega)$ and we have ${\partial_\nu \phi_k}_{|\Gamma}\in H^{1/2}(\Gamma)$. 

From now on, we fix $f \in H^{1/2}(\Gamma)$ and $\lambda \in {\Bbb C} \setminus \sigma (H)$ and we consider the problem
\bel{eq1}
\left\{ 
\begin{array}{rcll} 
(-{\rm i}\nabla+A)^2 u + Vu -\lambda u & = & 0, & \mbox{in}\ \Omega ,\\ 
u(x) & = & f(x),& x\in \Gamma.
\end{array}\right.
\ee
We start with two results related to the asymptotic behavior of solutions of \eqref{eq1} as $\lambda\to-\infty$.

\begin{lemma}\label{lem:Resolution}
For any $f \in H^{1/2}(\Gamma)$ and $\lambda \in {\Bbb C} \setminus \sigma (H)$, there exists a unique solution $u \in H^1(\Omega)$ to \eqref{eq1}
which can be written as 
\begin{equation}\label{eq:Sol-Series}
u_{\lambda} := u = \sum_{k \geq 1} {\alpha_{k} \over \lambda - \lambda_{k}}\,  \phi_{k},
\end{equation}
where for convenience we set
\begin{equation}\label{eq:Def-psi-alpha}
h_{k} := {\partial_\nu \phi_{k}}_{|\Gamma} , \qquad
\mbox{and}\qquad 
 \alpha_{k} :=\langle f,h_k \rangle.
\end{equation}
Moreover, we have
$$\|u_{\lambda}\|^2_{L^2(\Omega)} = \sum_{k \geq 1} {|\alpha_{k}|^2 \over |\lambda - \lambda_{k}|^2} \to 0 \qquad\mbox{as }\, \lambda \to -\infty.$$

\end{lemma}

\begin{proof}
Since $\lambda\notin\sigma(H)$, one can easily check that \eqref{eq1} admits a unique solution $u_\lambda \in H^1(\Omega)$. Moreover, $u_\lambda$ can be written in terms of the eigenvalues and eigenfunctions $\lambda_{k},\phi_{k}$. Indeed,  $u_\lambda \in L^2(\Omega)$ can be expressed in the Hilbert basis $(\phi_{k})_{k\geq 1}$ as
$$u_\lambda = \sum_{k\geq 1} (u_\lambda|\phi_{k})\, \phi_{k}\, $$
with $(\cdot,\cdot)$ the scalar product with respect to $L^2(\Omega)$. Since $u_\lambda\in H^1(\Omega)$ and $\Delta u_\lambda=-2iA\cdot\nabla u_\lambda+(-{\rm i}\ div(A) +|A|^2+V)u_\lambda\in L^2(\Omega)$, we have  $\nabla u_\lambda\in  H_{div}(\Omega)=\{v\in L^2(\Omega;\mathbb C^n):\ div(v)\in L^2(\Omega)\}$.
Thus, taking the scalar product of the first equation in \eqref{eq1} with $\phi_{k}$ and applying the Green formula  we obtain
$$\left\langle f,h_k\right\rangle = (\lambda - \lambda_{k})\,(u|\phi_{k}),$$
which yields the expression given by \eqref{eq:Sol-Series}.
The fact that $\|u_{\lambda}\| \to 0$ as $\lambda \to -\infty$ is a consequence of the fact that we may fix $c_{0} > \norm{V}_{L^\infty(\Omega)}$ large enough so that if $\lambda$ is real and such that $\lambda \leq - c_{0}$, we have $|\lambda - \lambda_{k}|^2 \geq  |c_0-\lambda_{k}|^2$ for all $k \geq 1$, and thus
$${|\alpha_{k}|^2 \over |\lambda - \lambda_{k}|^2} \leq  
{|\alpha_{k}|^2 \over  |c_0-\lambda_{k}|^2 } \, ,$$
so that we may apply Lebesgue's dominated convergence as $\lambda \to -\infty$.\end{proof}
\begin{lemma}\label{lem:Resolution:grad}
For  all $\lambda<-\norm{V}_{L^\infty(\Omega)}-6\norm{A}_{L^\infty(\Omega,\R^n)}^2$,   the solution $u_\lambda$ of \eqref{eq1} satisfies
\bel{lol1} \norm{\nabla u_\lambda}_{L^2(\Omega\setminus \Omega_1)}\leq C\norm{u_\lambda}_{L^2(\Omega_1)}\ee
with $C$ depending only on $\Omega$ and $\Omega_1$.

\end{lemma}
\begin{proof}
 Let us denote by $\chi\in\mathcal C^\infty_0(\Omega,\R)$ a function satisfying $\chi=1$ on $\Omega\setminus \Omega_1$. Then, since $\nabla u_\lambda\in  H_{div}(\Omega)$, multiplying \eqref{eq1} by $\chi^2\overline{u_\lambda}$ and applying the Green formula we obtain
\bel{in}\begin{aligned}0&=\int_\Omega (-{\rm i}\nabla +A)^2 u_\lambda \chi^2\overline{u_\lambda}dx+\int_\Omega(V-\lambda)\chi^2\abs{u_\lambda}^2dx\\
\ &= \int_\Omega |\chi\nabla u_\lambda|^2dx+2\int_\Omega (\chi\nabla u_\lambda)\cdot  \overline{\nabla\chi u_\lambda}dx\\
\ &\ \ +\int_\Omega [2{\rm i}(u_\lambda \chi)A\cdot \overline{u_\lambda\nabla\chi }+{\rm i}\chi u_\lambda A\cdot  \overline{\chi\nabla u_\lambda} +\chi\nabla u_\lambda\cdot \overline{{\rm i}A\chi u_\lambda}]dx+\int_\Omega(|A|^2+V-\lambda)\chi^2\abs{u_\lambda}^2dx.\end{aligned}\ee
Applying the Cauchy-Schwarz inequality we find
\[\begin{array}{l} \norm{\chi\nabla u_\lambda}_{L^2(\Omega)}^2+(-\norm{A}_{L^\infty(\Omega,\R^n)}^2-\norm{V}_{L^\infty(\Omega)}-\lambda)\norm{\chi u_\lambda}^2_{L^2(\Omega)}\\
\ \\
\leq 2\norm{u_\lambda\nabla \chi }_{L^2(\Omega)}\norm{\chi\nabla u_\lambda}_{L^2(\Omega)}+2\norm{A}_{L^\infty(\Omega)}\norm{\chi u_\lambda}_{L^2(\Omega)}\norm{u_\lambda\nabla \chi }_{L^2(\Omega)}+2\norm{A}_{L^\infty(\Omega)}\norm{\chi u_\lambda}_{L^2(\Omega)}\norm{\chi\nabla u_\lambda}_{L^2(\Omega)}\\
\ \\
\leq 4\norm{u_\lambda\nabla \chi }^2_{L^2(\Omega)}+{\norm{\chi\nabla u_\lambda}_{L^2(\Omega)}^2\over 4} +\norm{A}_{L^\infty(\Omega,\R^n)}^2\norm{\chi u_\lambda}_{L^2(\Omega)}^2+\norm{u_\lambda\nabla \chi }_{L^2(\Omega)}^2+4\norm{A}_{L^\infty(\Omega,\R^n)}^2\norm{\chi u_\lambda}_{L^2(\Omega)}^2\\
\ \\
\ \ \ +{\norm{\chi\nabla  u_\lambda}_{L^2(\Omega)}^2\over 4}.\end{array}\]
From this estimate, we deduce
\[{\norm{\chi\nabla u_\lambda}_{L^2(\Omega)}^2\over 2}+(-\norm{V}_{L^\infty(\Omega)}-6\norm{A}_{L^\infty(\Omega,\R^n)}^2-\lambda)\norm{\chi u_\lambda}^2_{L^2(\Omega)}\leq 5\norm{u_\lambda\nabla\chi }^2_{L^2(\Omega)}.\]
Using the fact that $\lambda<-\norm{V}_{L^\infty(\Omega)}-6\norm{A}_{L^\infty(\Omega,\R^n)}^2$, we obtain
\[\norm{\nabla u_\lambda}_{L^2(\Omega\setminus \Omega_1)}^2\leq\norm{\chi\nabla u_\lambda}_{L^2(\Omega)}^2\leq 10\norm{u_\lambda\nabla\chi }^2_{L^2(\Omega)}\leq 10\norm{\nabla\chi}_{L^\infty(\Omega)}^2\norm{ u_\lambda}^2_{L^2( \Omega_1)}\leq C\norm{ u_\lambda}^2_{L^2( \Omega_1)}.\]
From this estimate we deduce \eqref{lol1}.

\end{proof}

It is clear that the series \eqref{eq:Sol-Series} giving $u_{\lambda}$ in terms of $\alpha_{k},\lambda_{k}$ and $\phi_{k}$, converges only in $L^2(\Omega)$ and thus we cannot deduce an expression of the normal derivative $\partial_\nu u_{\lambda}$ in terms of $\alpha_{k}, \lambda_{k}$ and $h_{k}$. To avoid this difficulty, in a similar way to \cite{KKS}, we have the following lemma:

\begin{lemma}\label{lem:v-lambda-mu}
Let $f \in H^{1/2}(\Gamma)$ be fixed and for $\lambda,\mu \in {\Bbb C} \setminus \sigma (H)$ let $u_{\lambda}$ and $u_{\mu}$ be the solutions given by Lemma \ref{lem:Resolution}. If we set $v := v_{\lambda,\mu} := u_{\lambda} - u_{\mu}$, then
\begin{equation}\label{eq:v-Normal-Deriv}
\partial_\nu v  = \sum_{k\geq1} 
{(\mu - \lambda)\alpha_{k} \over (\lambda - \lambda_{k})(\mu - \lambda_{k})}\, h_{k}\, ,
\end{equation}
the convergence taking place in $H^{1/2}(\Gamma)$.
\end{lemma}

\begin{proof} Let $v_{\lambda,\mu} := u_{\lambda} - u_{\mu}$; One verifies that $v_{\lambda,\mu}$ solves
\begin{equation}\label{eq:v-lambda}
\left\{ 
\begin{array}{rcll} 
(-{\rm i}\nabla+A)^2 v_{\lambda,\mu} + V v_{\lambda,\mu} - \lambda v_{\lambda,\mu} & = & (\lambda - \mu)u_{\mu}, & \mbox{in}\ \Omega ,\\ 
v_{\lambda,\mu}(x) & = & 0,& x\in \Gamma.
\end{array}\right.
\end{equation}
Since $(u_{\mu}|\phi_{k}) = \alpha_{k}/(\mu - \lambda_{k})$, it follows that
$$v_{\lambda,\mu} = \sum_{k\geq1}{(\lambda - \mu) \alpha_{k}\over (\lambda_{k} - \lambda)(\mu - \lambda_{k}) }\, \phi_{k},$$
the convergence taking place in $D(H)$. Since the operator $v \mapsto \partial_\nu v$ is continuous from $D(H)$ into $H^{1/2}(\Gamma)$, the result of the lemma follows.
\end{proof}

The next lemma states essentially that if for $j=1$ or $j=2$ we have two magnetic potentials $A_j$, two electric potentials $V_j$ and $u_j := u_{j,\mu}$ solutions of
\begin{equation}\label{eq:u-m}
\left\{ 
\begin{array}{rcll} 
(-{\rm i}\nabla+A_j)^2 u_j + V_ju_j -\mu u_j & = & 0, & \mbox{in}\ \Omega ,\\ 
u_j(x) & = & f(x),& x\in \Gamma,
\end{array}\right.
\end{equation}
then $u_{1,\mu}$ and $u_{2,\mu}$ are {\it close\/} as $\mu \to -\infty$: in some sense the influence of the potentials $A_j$ and $V_{j}$ are dimmed when $\mu \to -\infty$. More precisely we have:

\begin{lemma}\label{lem:z-mu} 
Let $V_{j} \in L^\infty(\Omega,{\Bbb R})$ and $A_j\in W^{1,\infty}(\Omega,\R^n)$ be given for $j=1$ or $j=2$, and denote by $H_j$ the corresponding operator defined by \eqref{eq:Def-A-theta}. We assume that condition \eqref{t1a} is fulfilled.
For $f \in H^{1/2}(\Gamma)$ and  $\mu\in(-\infty, \mu_*)\subset {\Bbb C} \setminus \sigma (H) $, let $u_{j,\mu} := u_{j}$ be the solution of \eqref{eq:u-m}. Then  $z_{\mu} := u_{1,\mu} - u_{2,\mu}\in H^2(\Omega)$  converge to $0$ in $H^{2}(\Omega)$ as $\mu\to-\infty$. In particular $\partial_\nu z_{\mu} \to 0$ in $L^2(\Gamma)$ as $\mu \to -\infty$.
\end{lemma}

\begin{proof}
Since the trace map $v\mapsto \partial_\nu v$ is  continuous from $H^{2}(\Omega)$ to $L^2(\Gamma)$, it is enough to show that $z_\mu\in H^2(\Omega)$ and  $\|z_{\mu}\|_{H^{2}(\Omega)} \to 0$ when $\mu \to -\infty$.  We fix $\mu<\mu_*$ with $\mu_*<-\norm{V}_{L^\infty(\Omega)}-6\norm{A}_{L^\infty(\Omega,\R^n)}^2$ less than  the constants given by Lemma \ref{lem:Resolution}  for $A=A_j$, $V=V_j$, $j=1,2$. Without lost of generality we assume that $H_j-\mu_*$ is  positive, $j=1,2$.
One verifies that $z_{\mu}$ solves the equation
\begin{equation}\label{eq:z-mu}
\left\{ 
\begin{array}{rcll} 
(-{\rm i}\nabla+A_1)^2 z_{\mu} + V_1z_{\mu} - \mu z_{\mu} & = & -2{\rm i}(A_2-A_1)\cdot\nabla u_{2,\mu}+(p_2-p_1)u_{2,\mu}, & \mbox{in}\ \Omega,\\ 
z_{\mu}(x) & = & 0,& x\in \Gamma
\end{array}\right.
\end{equation}
with $p_j=-{\rm i}div(A_j)+\abs{A_j}^2+V_j$, $j=1,2$.
That is, denoting by $R_{1,\mu} = (H_1 - \mu I)^{-1}$ the resolvent of the operator $H_1:= (-{\rm i}\nabla+A_1)^2  + V_1$, we have $$z_{\mu} = R_{1,\mu}(-2{\rm i}(A_2-A_1)\cdot\nabla u_{2,\mu}+(p_2-p_1)u_{2,\mu})=\sum_{k=1}^{+\infty}{(w_\mu,\phi_{1,k})\over (\lambda_{1,k}-\mu)}\phi_{1,k}$$
with $w_\mu=-2{\rm i}(A_2-A_1)\cdot\nabla u_{2,\mu}+(p_2-p_1)u_{2,\mu}$ and $(\lambda_{1,k})_{k\geq1}$, $(\phi_{1,k})_{k\geq1}$ respectively the eigenvalues of $H_1$ and an Hilbertian basis of eigenfunctions associated to these eigenvalues. Since $w_\mu\in L^2(\Omega)$, $z_\mu$ is lying in $D(H_1)$ and by the same way in $H^2(\Omega)$. It remains to show that  $\|z_{\mu}\|_{H^{2}(\Omega)} \to 0$ when $\mu \to -\infty$.  Since $D(H_1)$ embedded continuously into $H^2(\Omega)$ there exists a generic constant $C$ depending on $A_1$, $V_1$ and $\Omega$ such that
\[\norm{z_\mu}_{H^2(\Omega)}^2\leq C\sum_{k=1}^{\infty}\abs{\lambda_{1,k}-\mu_*}^2\abs{(z_\mu,\phi_{1,k})}^2\leq C\norm{w_\mu}_{L^2(\Omega)}.\]
On the other hand, condition \eqref{t1a} implies
\[\norm{w_\mu}_{L^2(\Omega)}\leq C(\norm{\nabla u_{2,\mu}}_{L^2(\Omega\setminus\Omega_1)}+\norm{ u_{2,\mu}}_{L^2(\Omega)})\]
with $C$ independent of $\mu$. Then, according to Lemma \ref{lem:Resolution} and \eqref{lol1}, we obtain
\[\limsup_{\mu\to-\infty}\norm{w_\mu}_{L^2(\Omega)}\leq C\limsup_{\mu\to-\infty}\norm{ u_{2,\mu}}_{L^2(\Omega)}=0.\]
Thus, we have
$$\limsup_{\mu\to-\infty}\norm{z_\mu}_{L^2(\Gamma)}\leq C \limsup_{\mu\to-\infty}\|z_{\mu}\|_{H^{2}(\Omega)}\leq C\limsup_{\mu\to-\infty}\norm{w_\mu}_{L^2(\Omega)}=0.$$
This completes the proof.
\end{proof}

Armed with these results, we will prove Theorem \ref{thm-1}  by using some asymptotic properties of solutions of \eqref{eq1} with respect to $\lambda$. For this purpose, like in \cite{CS,I,KKS} we use representation formulas that will allow  us to make a connection between the boundary spectral data and the potentials $A$ and $V$.

\section{Representation formulas}
 From now on, for all $x=(x_1,\ldots,x_n)\in\C^n$ and $y=(y_1,\ldots,y_n)\in\C^n$, we denote by $x\cdot y$ the quantity 
\[x\cdot y=\sum_{k=1}^nx_ky_k\]
and for all $x\in\R^n$ we denote by $x^\bot$ the subspace of $\R^n$ defined by $\{y\in\R^n:\ y\cdot x=0\}$. Moreover, we set $A_j\in\mathcal C^1(\overline{\Omega},\R^n)$, $V_j\in L^\infty(\Omega,\R)$, $j=1,2$, and we assume that condition \eqref{t1a} is fulfilled. For $j=1,2$ and $\lambda\in \mathbb C\setminus \R$, we associate to the problem
\bel{eq:lambda}
\left\{ 
\begin{array}{rcll} 
(-{\rm i}\nabla+A_j)^2 u_j + V_ju_j -\lambda u_j & = & 0, & \mbox{in}\ \Omega ,\\ 
u_j(x) & = & f(x),& x\in \Gamma
\end{array}\right.\ee
the Dirichlet-to-Neumann map
\[\Lambda_{j,\lambda}:H^{{1\over 2}}(\partial\Omega)\ni f\mapsto{(\partial_\nu +{\rm i}A_j\cdot\nu)u_{j,\lambda}}_{|\Gamma},\]
where $u_{j,\lambda}$ solves \eqref{eq:lambda}.
The goal of this section is to apply  the Dirichlet-to-Neumann maps $\Lambda_{j,\lambda}$ to some suitable ansatzs associated with \eqref{eq:lambda} in order to get two representation formulas involving the magnetic potentials $A_j$ and the electric potentials $V_j$, $j=1,2$.  A similar approach was  developed by \cite{I} and \cite{CS,KKS} used the representation of \cite{I}. The idea is to establish the link between the electric and magnetic potentials and the boundary spectral data by mean of an expression involving  the Dirichlet-to-Neumann maps $\Lambda_{1,\lambda}$, $\Lambda_{2,\lambda}$.
We start with two general representation formulas, stated in the next subsection,  where some properties of the ansatzs will not be completely specified. This will allow us to clarify the main goal of these formulas. Then, in Subsection 3.2 we will introduce the remaining properties of our ansatzs and establish some asymptotic properties from our representations which will be one of the main points of our analysis. 

\subsection{General representation formulas}

In this subsection we introduce the first formulation of two representation formulas involving respectively the Dirichlet-to-Neumann maps $\Lambda_{1,\lambda}$, $\Lambda_{2,\lambda}$ and some  ansatzs associated with problem \eqref{eq:lambda}. In \cite{I}, Isozaki considered such  formulas for Schr\"odinger operators $-\Delta+V$ with  an electric potential $V$, in other words for Schr\"odinger operators with a  variable coefficient of order zero. In our case we need to extend this strategy to Schr\"odinger operators with
both magnetic and electric potentials, which means an extension to Schr\"odinger  operators with variable coefficients of order zero and one. In addition, we need to consider ansatzs that allow to recover both the magnetic field and the electric potential. Therefore, we consider some ansatzs, associated with  \eqref{eq:lambda}, of the form
\bel{anz}\Phi_{j,\lambda}(x)=e^{\zeta_j\cdot x}g_j(x),\quad \zeta_j\in\mathbb C^n,\ x\in\Omega,\ j=1,2\ee
with $\zeta_j$  satisfying $\zeta_j\cdot\zeta_j=-\lambda$ and with $g_1$ and $g_2$ respectively a solution of
\bel{eqeq}\zeta_1\cdot\nabla g_1+({\rm i}\zeta_1\cdot A_{1,\sharp})g_1=0,\quad \zeta_2\cdot\nabla g_2-({\rm i}\zeta_2\cdot A_{2,\sharp})g_2=0\ee
with $A_{j,\sharp}$ some smooth function close to the magnetic potential  $A_j$, $j=1,2$. 
More precisely, we fix  $\eta_1,\eta_2 \in \mathbb S^{n-1}=\{y \in \R^n,\ \abs{y} = 1 \}$ and we define $A_{j,\sharp}\in \mathcal C_0^\infty(\R^n,\R^n)$, $j=1,2$, some smooth approximations on $\overline{\Omega}$ of $A_j$. Then, we set
$\zeta_1={\rm i}\sqrt{\lambda}\eta_1$, $\zeta_2=-{\rm i}\sqrt{\lambda}\eta_2$ and we consider solutions of the transport equations \eqref{eqeq} given by
$$g_1(x):=e^{{\rm i}\psi_1(x)},\quad   g_2(x):=b_2(x)e^{-{\rm i}\psi_2(x)},\quad \psi_j(x):=-\int_{-\infty}^0 \eta_j\cdot A_{j,\sharp}(x+s\eta_j)ds,\quad  \eta_2\cdot\nabla b_2(x)=0,\quad x\in\R^n.$$
Therefore, we consider ansatzs associated with \eqref{eq:lambda} taking the form
\bel{jt1}
\Phi_{1,\lambda}(x):=e^{{\rm i}  \sqrt{\lambda} \eta_1 \cdot x}e^{{\rm i}\psi_1(x)},\quad \Phi_{2,\lambda}(x):=e^{-{\rm i} \sqrt{\lambda} \eta_2 \cdot x}b_2(x)e^{-{\rm i}\psi_2(x)},\quad x\in\Omega.
\ee
We assume in addition  that  $b_2\in W^{2,\infty}(\R^n)$ and we recall that $\psi_j$ solves the equation
\[\eta_j\cdot\nabla\psi_j(x)=-\eta_j\cdot A_{j,\sharp}, \ \ j=1,2,\quad x\in\R^n.\]
For the time being, we consider general ansatzs of the form \eqref{jt1} with the properties describe above. Additional information about the parameter $\lambda$, the function $A_{j,\sharp}$, the vector $\eta_j$, $j=1,2$,  and the function $b_2$ will be given in Subsection 3.2. In a similar  way to \cite{FKSU,KLU,KU,NSU1,Sa1,Su}, in the construction of our ansatzs  we consider some smooth approximations of the magnetic potentials instead of the magnetic potentials to obtain sufficiently smooth functions $\Phi_{j,\lambda}$, $j=1,2$. Using this approach, we can  weaken the regularity  assumption imposed on admissible magnetic potential from $W^{3,\infty}(\Omega)$ to $\mathcal C^1(\overline{\Omega})$.
Further, for $j=1,2$, we put
\bel{kt1}
S_j(\lambda,\eta_1,\eta_2)=\left \langle \Lambda_{j,\lambda} \Phi_{1,\lambda} , \overline{\Phi_{2,\lambda}}\right\rangle=\int_\Gamma(\Lambda_{j,\lambda} \Phi_{1,\lambda})\Phi_{2,\lambda}(x) d\sigma(x).
\ee
In other words, we apply  $\Lambda_{j,\lambda}$, $j=1,2$,  to ansatzs of the form \eqref{anz} with $\zeta_1={\rm i} \sqrt{\lambda} \eta_1 $, $\zeta_2=-{\rm i}  \sqrt{\lambda}\eta_2$, $g_1=e^{{\rm i} \psi_1}$ and $g_2=b_2e^{-{\rm i} \psi_2}$. We recall that  quantities similar to $S_1$ and $S_2$ have also been used by \cite{FKSU,I,KKS,KLU,KU,NSU1,Sa1,Su}. Let us also mention that, like in \cite{I,KKS}, the ansatzs  \eqref{jt1} do not depend on the potential $V_1$ and $V_2$ which are coefficients of order zero of the equation \eqref{eq:lambda}. On the other hand, the ansatzs  \eqref{jt1} depend on the magnetic potentials $A_1$ and $A_2$ which are coefficients of order one of the equation \eqref{eq:lambda}. By modifying the construction of 
\cite{I,KKS} with the new expression $g_j$, $j=1,2$, we will extend the approach of \cite{I,KKS} to Shr\"odinger operators with magnetic potentials. 
From now on, for the sake of simplicity we will systematically omit the subscripts $\lambda$ in $\Phi_{j,\lambda}$, $j=1,2$,  in the remaining of this text. In view of determining the behavior of $S_1-S_2$,  as $\im\lambda\to+\infty$, we introduce the following representations associated with  $S_1$ and $S_2$.

\begin{proposition}
\label{l1} 
For all  $\lambda \in \C\setminus\R$ and $\eta_j \in\mathbb S^{n-1}$, $j=1,2$, the scalar products $S_j(\lambda,\eta_1,\eta_2)$ have the following expression 
\begin{eqnarray}\label{l1a}
&&S_1(\lambda,\eta_1,\eta_2)\cr
&&= 2\sqrt{\lambda}\int_\Omega \eta_2\cdot (A_1-A_{2,\sharp})e^{{\rm i}\sqrt{\lambda}({\eta}_1-{\eta}_2)\cdot x} b_2 e^{{\rm i}(\psi_1(x)-\psi_2(x))}dx\cr
&&\ +\int_\Omega (V_1-q_{12})e^{{\rm i}\sqrt{\lambda}({\eta}_1-{\eta}_2)\cdot x} b_2 e^{{\rm i}(\psi_1(x)-\psi_2(x))}dx\cr
&&\ -{\rm i} \int_{\Gamma}   e^{{\rm i}\sqrt{\lambda}({\eta}_1-{\eta}_2)\cdot x} e^{{\rm i}(\psi_1(x)-\psi_2(x))}(b_2\sqrt{\lambda} \eta_2+b_2\nabla\psi_2+{\rm i}\nabla b_2 +b_2A_1)\cdot \nu d\sigma(x)\cr
&&\ -\int_\Omega \left[(H_1-\lambda)^{-1}\left(2\sqrt{\lambda}\eta_1\cdot(A_1-A_{1,\sharp})+q_{11}\right)\Phi_1\right]\left(2\sqrt{\lambda}\eta_2\cdot (A_1-A_{2,\sharp})b_2+V_1b_2-q_{12}\right)e^{-{\rm i}\sqrt{\lambda}{\eta}_2 \cdot x}e^{-{\rm i}\psi_2}dx,\cr &&\end{eqnarray}

\begin{eqnarray}\label{l1b}
&&S_2(\lambda,\eta_1,\eta_2) \cr
&&=\int_\Omega \left[2\sqrt{\lambda}\eta_2\cdot (A_2-A_{2,\sharp})+V_2-q_{22}\right]e^{{\rm i}\sqrt{\lambda}({\eta}_1-{\eta}_2)\cdot x} b_2 e^{{\rm i}(\psi_1(x)-\psi_2(x))}dx\cr
&&\ -{\rm i} \int_{\partial\Omega}   e^{{\rm i}\sqrt{\lambda}({\eta}_1-{\eta}_2)\cdot x} e^{{\rm i}(\psi_1(x)-\psi_2(x))}(b_2\sqrt{\lambda} \eta_2+b_2\nabla\psi_2+{\rm i}\nabla b_2 +b_2A_2)\cdot \nu d\sigma(x)\cr
&&\ -\int_\Omega \left[(H_2-\lambda)^{-1}\left(2\sqrt{ \lambda}\eta_1\cdot(A_2-A_{1,\sharp})+q_{21}\right)\Phi_1\right](2\sqrt{\lambda}\eta_2\cdot (A_2-A_{2,\sharp})b_2+V_2b_2-q_{22})e^{-{\rm i}\sqrt{\lambda}{\eta}_2 \cdot x}e^{-{\rm i}\psi_2}dx.\cr &&\end{eqnarray}

Here we denote by $q_{11}$, $q_{12}$, $q_{21}$, $q_{22}$ the expressions
\[q_{11}=-{\rm i}div (A_1)+\abs{A_1}^2+V_1(x)+2 A_1\cdot\nabla\psi_1-{\rm i}\Delta \psi_1+\abs{\nabla\psi_1}^2,\]
\[q_{12}=\Delta b_2-2{\rm i}\nabla\psi_2\cdot\nabla b_2-2{\rm i}\nabla b_2\cdot A_1+\left(-{\rm i}\Delta\psi_2-\abs{\nabla\psi_2}^2-2\nabla\psi_2\cdot A_1-{\rm i}div(A_1)-\abs{A_1}^2\right)b_2,\]
\[q_{21}=-{\rm i}div (A_2)+\abs{A_2}^2+V_2(x)+2 A_2\cdot\nabla\psi_1-{\rm i}\Delta \psi_1+\abs{\nabla\psi_1}^2,\]
\[q_{22}=\Delta b_2-2{\rm i}\nabla\psi_2\cdot\nabla b_2-2{\rm i}\nabla b_2\cdot A_2+\left(-{\rm i}\Delta\psi_2-\abs{\nabla\psi_2}^2-2\nabla\psi_2\cdot A_2-{\rm i}div(A_2)-\abs{A_2}^2\right)b_2.\]
Moreover, $H_j$, $j=1,2,$  denotes the selfadjoint operator $(-{\rm i}\nabla+A_j)  + V_j$ acting on $L^2(\Omega)$ with domain 
\[D(H_j)=\{v\in H^1_0(\Omega):\  (-{\rm i}\nabla+A_j) v\in L^2(\Omega)\}.\]
\end{proposition}

Note that formulas \eqref{l1a}-\eqref{l1b} contain expressions involving the magnetic potentials $A_1$, $A_2$ and the electric potentials  $V_1$, $V_2$, expressions on the boundary $\partial\Omega$ and expressions described by the resolvent  $(H_j-\lambda)^{-1}$, $j=1,2$. Using condition \eqref{t1a} one can check that the expressions on $\partial\Omega$ of $S_1$ and $S_2$ coincide and applying the decay of the resolvent $(H_j-\lambda)^{-1}$, $j=1,2$, as $\im \lambda\to+\infty$ we will show in the next subsection that, for some suitable choice of our ansatzs, the expressions
$$-\int_\Omega \left[(H_1-\lambda)^{-1}\left(2\sqrt{\lambda}\eta_1\cdot(A_1-A_{1,\sharp})+q_{11}\right)\Phi_1\right]\left(2\sqrt{\lambda}\eta_2\cdot (A_1-A_{2,\sharp})b_2+V_1b_2-q_{12}\right)e^{-{\rm i}\sqrt{\lambda}{\eta}_2 \cdot x}e^{-{\rm i}\psi_2}dx,$$
$$-\int_\Omega \left[(H_2-\lambda)^{-1}\left(2\sqrt{ \lambda}\eta_1\cdot(A_2-A_{1,\sharp})+q_{21}\right)\Phi_1\right](2\sqrt{\lambda}\eta_2\cdot (A_2-A_{2,\sharp})b_2+V_2b_2-q_{22})e^{-{\rm i}\sqrt{\lambda}{\eta}_2 \cdot x}e^{-{\rm i}\psi_2}dx,$$
vanish as $\im \lambda\to+\infty$. Thus, what will remain in the asymptotic expansion of $S_1-S_2$,  as $\im\lambda\to+\infty$, will be two expressions involving $A_1-A_2$ and $V_1-V_2$. These two expressions, that will be given in the next subsection, are one of the main ingredients in our proof. The remaining of this subsection will be devoted to the proof of Proposition \ref{l1}.\\
\textbf{Proof of Proposition \ref{l1}.} Let us first remark that the expressions \eqref{l1a}-\eqref{l1b} correspond to some asymptotic expansion of the expression $S_j$, $j=1,2$, with respect to $\sqrt{\lambda}$ \footnote{This statement will be clarified in the next subsection where we will give additional information about the parameter $\lambda$ and the vectors $\eta_1$, $\eta_2$.}. We will prove  \eqref{l1a}-\eqref{l1b} by combining properties of the ansatzs \eqref{anz}, with properties of solutions of \eqref{eq:lambda} when $f=\Phi_1$.  This proof will be divided into two steps, first for $S_1$ then for $S_2$. We start by showing that for $j=1,2$ and $f=\Phi_1$ problem \eqref{eq:lambda} admits a unique solution $u_j\in H^2(\Omega)$ taking the form
\bel{l1d}
u_1=\Phi_1-(H_1-\lambda)^{-1}\left[2\sqrt{\lambda}\eta_1\cdot (A_1-A_{1,\sharp})+q_{11}\right]\Phi_1,
\ee
\bel{l1i}
u_2=\Phi_1-(H_2-\lambda)^{-1}\left[2\sqrt{ \lambda}\eta_1\cdot(A_2-A_{1,\sharp})+q_{21}\right]\Phi_1.
\ee
Then, combining these formulas with the properties of the ansatzs \eqref{anz} and applying the Green formula, we derive \eqref{l1a}-\eqref{l1b}.

We start with  the expression of $S_1(\lambda,\eta_1,\eta_2)$. Let us first prove \eqref{l1d}. Recall that 
\[(-{\rm i}\nabla+A_1)^2 u+V_1u-\lambda u=-\Delta u -2{\rm i} A_1\cdot \nabla u+qu -\lambda u\]
with $q(x)=-{\rm i}div (A_1)(x)+\abs{A_1(x)}^2+V_1(x)$. Therefore, in light of \eqref{jt1} we have
\[\begin{array}{l}(-{\rm i}\nabla+A_1)^2 \Phi_1+V_1\Phi_1-\lambda \Phi_1\\
=(\lambda+2\sqrt{\lambda}\eta_1\cdot\nabla\psi_1-{\rm i}\Delta \psi_1+\abs{\nabla\psi_1}^2)\Phi_1+(2\sqrt{\lambda} \eta_1\cdot A_1+2 A_1\cdot\nabla\psi_1)\Phi_1+q\Phi_1-\lambda\Phi_1\\
=2\sqrt{\lambda}(\eta_1\cdot\nabla\psi_1+\eta_1\cdot A_1)\Phi_1+ q_{11}\Phi_1\end{array}\]
with $q_{11}=q+2 A_1\cdot\nabla\psi_1-{\rm i}\Delta \psi_1+\abs{\nabla\psi_1}^2$. On the other hand, since $\psi_1$ satisfies $\eta_1\cdot\nabla\psi_1+\eta_1\cdot A_{1,\sharp}=0$, we deduce that
\bel{l1c}
(-{\rm i}\nabla+A_1)^2 \Phi_1+V_1\Phi_1-\lambda \Phi_1=\left[2\sqrt{\lambda}\eta_1\cdot (A_1-A_{1,\sharp})+q_{11}\right]\Phi_1.
\ee
Now consider $u_1$ the solution of
\[\left\{ 
\begin{array}{rcll} 
(-{\rm i}\nabla+A_1)^2 u_1 + V_1u_1 -\lambda u_1 & = & 0, & \mbox{in}\ \Omega ,\\ 
u_1(x) & = & \Phi_1(x),& x\in \partial\Omega.
\end{array}\right.\]
Note that, with our assumptions one can check that $D(H_1)=H^1_0(\Omega)\cap H^2(\Omega)$. In view of \eqref{l1c}, we can split $u_1$ into two terms $u_1=\Phi_1+v_1$ with $v_1$ the solution of 
\[\left\{ 
\begin{array}{rcll} 
(-{\rm i}\nabla+A_1)^2 v_1 + V_1v_1 -\lambda v_1 & = & -\left[2\sqrt{\lambda}\eta_1\cdot (A_1-A_{1,\sharp})+q_{11}\right]\Phi_1, & \mbox{in}\ \Omega ,\\ 
v_1(x) & = & 0,& x\in \partial\Omega.
\end{array}\right.\]
Then, $u_1\in H^2(\Omega)$ take the form \eqref{l1d}. Using this formula we will complete the proof of \eqref{l1a}. Since
\bel{l1f}
S_1 =\int_{\partial\Omega}(\partial_\nu+{\rm i}A_1\cdot\nu) u_1(x) e^{-{\rm i} \sqrt{\lambda}{\eta}_2\cdot x}b_2e^{-{\rm i}\psi_2(x)}d \sigma(x),
\ee
from \eqref{kt1}, applying Green formula,
we get 
\bel{l1g}
\begin{aligned}S_1&=
\int_\Omega div\left( (\nabla+{\rm i}A_1(x)) u_1(x) e^{-{\rm i} \sqrt{\lambda}{\eta}_2\cdot x}b_2e^{-{\rm i}\psi_2(x)}\right)  dx\\
\ &= \int_\Omega (\nabla+{\rm i}A_1)^2u_1e^{-{\rm i} \sqrt{\lambda}{\eta}_2\cdot x}b_2e^{-{\rm i}\psi_2}dx+\int_\Omega (\nabla+{\rm i}A_1)u_1\cdot (\nabla-{\rm i}A_1)e^{-{\rm i} \sqrt{\lambda}{\eta}_2\cdot x}b_2e^{-{\rm i}\psi_2}dx.\end{aligned}
\ee
 Doing the same with the second term on the right hand side of this formula, we find out that
\beas
& & \int_\Omega (\nabla+{\rm i}A_1)u_1\cdot (\nabla-{\rm i}A_1)e^{-{\rm i} \sqrt{\lambda}{\eta}_2\cdot x}b_2e^{-{\rm i}\psi_2}dx \\
& = & -{\rm i}\int_{\Gamma}  u_1(x) e^{-{\rm i}\sqrt{\lambda}{\eta}_2\cdot x} e^{-{\rm i}\psi_2}(\sqrt{\lambda} b_2\eta_2+b_2\nabla\psi_2+{\rm i}\nabla b_2 +b_2A_1)\cdot \nu d\sigma(x)\\&\ &
- \int_\Omega u_1(x) (\nabla-{\rm i}A_1)^2e^{-{\rm i}\sqrt{\lambda}{\eta}_2 \cdot x}b_2e^{-{\rm i}\psi_2}dx.
\eeas
In light of \eqref{jt1} and the identity ${u_1}_{\vert\Gamma}=\Phi_{1}$, this entails
\[\begin{array}{l}
 \int_\Omega (\nabla+{\rm i}A_1)u_1\cdot (\nabla-{\rm i}A_1)e^{-{\rm i} \sqrt{\lambda}{\eta}_2\cdot x}b_2e^{-{\rm i}\psi_2}dx \\
 =  -{\rm i} \int_{\Gamma}   e^{{\rm i}\sqrt{\lambda}({\eta}_1-{\eta}_2)\cdot x} e^{{\rm i}(\psi_1(x)-\psi_2(x))}(\sqrt{\lambda} b_2\eta_2+b_2\nabla\psi_2+{\rm i}\nabla b_2 +b_2A_1)\cdot \nu d\sigma(x)\\
\ \ \ -
\int_\Omega u_1(x) (\nabla-{\rm i}A_1)^2e^{-{\rm i}\sqrt{\lambda}{\eta}_2 \cdot x}b_2e^{-{\rm i}\psi_2}dx. 
\end{array}\]
Moreover, one can check that
\[(\nabla-{\rm i}A_1)^2e^{-{\rm i}\sqrt{\lambda}{\eta}_2 \cdot x}b_2e^{-{\rm i}\psi_2}=\left(-\lambda b_2-2\sqrt{\lambda}(\eta_2\cdot\nabla\psi_2+A_1\cdot\eta_2)b_2-2{\rm i}\sqrt{\lambda}\eta_2\cdot \nabla b_2+q_{12}\right)e^{-{\rm i}\sqrt{\lambda}{\eta}_2 \cdot x}e^{-{\rm i}\psi_2}\]
with $q_{12}=\Delta b_2-2{\rm i}\nabla\psi_2\cdot\nabla b_2-2{\rm i}\nabla b_2\cdot A_1+\left(-{\rm i}\Delta\psi_2-\abs{\nabla\psi_2}^2-2\nabla\psi_2\cdot A_1-{\rm i}div(A_1)-\abs{A_1}^2\right)b_2$. Combining this with the fact that $\psi_2$ satisfies $\eta_2\cdot\nabla\psi_2+\eta_2\cdot A_{2,\sharp}=0$ and $b_2$ solves $\eta_2\cdot\nabla b_2=0$, we deduce that
\[(\nabla-{\rm i}A_1)^2e^{-{\rm i}\sqrt{\lambda}{\eta}_2 \cdot x}b_2e^{-{\rm i}\psi_2}=\left([-\lambda-2\sqrt{\lambda}\eta_2\cdot (A_1-A_{2,\sharp})]b_2+q_{12}\right)e^{-{\rm i}\sqrt{\lambda}{\eta}_2 \cdot x}e^{-{\rm i}\psi_2}.\]
Therefore, we find
\[\begin{array}{l}
 \int_\Omega (\nabla+{\rm i}A_1)u_1\cdot (\nabla-{\rm i}A_1)e^{-{\rm i} \sqrt{\lambda}{\eta}_2\cdot x}b_2e^{-{\rm i}\psi_2}dx \\
 =  -{\rm i} \int_{\Gamma}   e^{{\rm i}\sqrt{\lambda}({\eta}_1-{\eta}_2)\cdot x} e^{{\rm i}(\psi_1(x)-\psi_2(x))}(\sqrt{\lambda} b_2\eta_2+b_2\nabla\psi_2+{\rm i}\nabla b_2 +b_2A_1)\cdot \nu d\sigma(x)\\
\ \ \ -
\int_\Omega u_1(x) \left(-\lambda b_2-2\sqrt{\lambda}\eta_2\cdot (A_1-A_{2,\sharp})b_2+q_{12}\right)e^{-{\rm i}\sqrt{\lambda}{\eta}_2 \cdot x}e^{-{\rm i}\psi_2}dx.
\end{array}\]
Then, from \eqref{l1d} we get
\begin{eqnarray}\label{l1h}
 &&\int_\Omega (\nabla+{\rm i}A_1)u_1\cdot (\nabla-{\rm i}A_1)e^{-{\rm i} \sqrt{\lambda}{\eta}_2\cdot x}b_2e^{-{\rm i}\psi_2}dx \cr
 &&=  -{\rm i} \int_{\Gamma}   e^{{\rm i}\sqrt{\lambda}({\eta}_1-{\eta}_2)\cdot x} e^{{\rm i}(\psi_1(x)-\psi_2(x))}(\sqrt{\lambda} b_2\eta_2+b_2\nabla\psi_2+{\rm i}\nabla b_2 +b_2A_1)\cdot \nu d\sigma(x)\cr
&&\ \ \ +\lambda \int_\Omega u_1e^{-{\rm i}\sqrt{\lambda}{\eta}_2 \cdot x}b_2e^{-{\rm i}\psi_2}dx+2\sqrt{\lambda}\int_\Omega \eta_2\cdot (A_1-A_{2,\sharp})e^{{\rm i}\sqrt{\lambda}({\eta}_1-{\eta}_2)\cdot x} b_2 e^{{\rm i}(\psi_1(x)-\psi_2(x))}dx\cr
&&\ \ \ -\int_\Omega q_{12}e^{{\rm i}\sqrt{\lambda}({\eta}_1-{\eta}_2)\cdot x} b_2 e^{{\rm i}(\psi_1(x)-\psi_2(x))}dx\cr
&&\ \ \ -\int_\Omega [(H_1-\lambda)^{-1}\left(2\sqrt{\lambda}\eta_1\cdot (A_1-A_{1,\sharp})+q_{11}\right)\Phi_1]\left(2\sqrt{\lambda}\eta_2\cdot (A_1-A_{2,\sharp})b_2-q_{12}\right)e^{-{\rm i}\sqrt{\lambda}{\eta}_2 \cdot x}e^{-{\rm i}\psi_2}dx.\cr
&&
\end{eqnarray}
Next, taking into account   the fact that
$(\nabla+{\rm i}A_1)^2u_1= (V_1 -\lambda) u_1$ in $\Omega$, we obtain
\[\begin{array}{ll}\int_\Omega (\nabla+{\rm i}A_1)^2u_1e^{-{\rm i} \sqrt{\lambda}{\eta}_2\cdot x}b_2e^{-{\rm i}\psi_2}dx&=\int_\Omega (V_1 -\lambda) u_1e^{-{\rm i} \sqrt{\lambda}{\eta}_2\cdot x}b_2e^{-{\rm i}\psi_2}dx\\
\ \\
\ &=-\lambda\int_\Omega  u_1e^{-{\rm i} \sqrt{\lambda}{\eta}_2\cdot x}b_2e^{-{\rm i}\psi_2}dx+\int_{\Omega}   V_1e^{{\rm i}\sqrt{\lambda}({\eta}_1-{\eta}_2)\cdot x} b_2 e^{{\rm i}(\psi_1(x)-\psi_2(x))}dx\\
\ \\
\ &\ \ \ -\int_\Omega V_1\left[(H_1-\lambda)^{-1}\left(2\sqrt{\lambda}\eta_1\cdot (A_1-A_{1,\sharp})+q_{11}\right)\Phi_1\right]e^{-{\rm i} \sqrt{\lambda}{\eta}_2\cdot x}b_2e^{-{\rm i}\psi_2}dx.\end{array}\]
Finally, we deduce \eqref{l1a} from \eqref{l1g}-\eqref{l1h}. 

Now let us consider \eqref{l1b}. For this purpose, we start by proving formula \eqref{l1i}. In a similar way to \eqref{l1d}, we have
\[(-{\rm i}\nabla+A_2)^2 \Phi_1+V_2\Phi_1-\lambda \Phi_1
=2\sqrt{\lambda}(\eta_1\cdot\nabla\psi_1+A_2\cdot \eta_1)\Phi_1+ q_{21}\Phi_1\]
with $q_{21}=-{\rm i}div (A_2)+\abs{A_2}^2+V_2(x)+2 A_2\cdot\nabla\psi_1-{\rm i}\Delta \psi_1+\abs{\nabla\psi_1}^2$. Then, since $\psi_1$ is a solution of $\eta_1\cdot\nabla\psi_1+\eta_1\cdot A_{1,\sharp}=0$, we deduce that
\[(-{\rm i}\nabla+A_2)^2 \Phi_1+V_2\Phi_1-\lambda \Phi_1=\left(2\sqrt{ \lambda}\eta_1\cdot(A_2-A_{1,\sharp})+q_{21}\right)\Phi_1.\]
Moreover, one can check that the solution  $u_2$  of
\[\left\{ 
\begin{array}{rcll} 
(-{\rm i}\nabla+A_2)^2 u_2 + V_2u_2 -\lambda u_2 & = & 0, & \mbox{in}\ \Omega ,\\ 
u_2(x) & = & \Phi_1(x),& x\in \partial\Omega
\end{array}\right.\]
is given by \eqref{l1i}. Repeating our previous arguments, we deduce 
\bel{l1j}S_2= \int_\Omega (\nabla+{\rm i}A_2)^2u_2e^{-{\rm i} \sqrt{\lambda}{\eta}_2\cdot x}b_2e^{-{\rm i}\psi_2}dx+\int_\Omega (\nabla+{\rm i}A_2)u_2\cdot (\nabla-{\rm i}A_2)e^{-{\rm i} \sqrt{\lambda}{\eta}_2\cdot x}b_2e^{-{\rm i}\psi_2}dx.\ee
On the other hand, using the fact that $\psi_2$ is a solution of the equation $\eta_2\cdot\nabla\psi_2+\eta_2\cdot A_{2,\sharp}=0$, we get
\begin{eqnarray}
 &&\int_\Omega (\nabla+{\rm i}A_2)u_2\cdot (\nabla-{\rm i}A_2)e^{-{\rm i} \sqrt{\lambda}{\eta}_2\cdot x}b_2e^{-{\rm i}\psi_2}dx \cr
&& =  -{\rm i} \int_{\Gamma}   e^{{\rm i}\sqrt{\lambda}({\eta}_1-{\eta}_2)\cdot x} e^{{\rm i}(\psi_1(x)-\psi_2(x))}(\sqrt{\lambda} b_2\eta_2+b_2\nabla\psi_2+{\rm i}\nabla b_2 +b_2A_2)\cdot \nu d\sigma(x)\cr
&&\ \ \ -
\int_\Omega u_2(x) \left(-\lambda b_2-2\sqrt{\lambda}\eta_2\cdot (A_2-A_{2,\sharp})b_2+q_{22}\right)e^{-{\rm i}\sqrt{\lambda}{\eta}_2 \cdot x}e^{-{\rm i}\psi_2}dx
\end{eqnarray}
with $q_{22}=\Delta b_2-2{\rm i}\nabla\psi_2\cdot\nabla b_2-2{\rm i}\nabla b_2\cdot A_2+\left(-{\rm i}\Delta\psi_2-\abs{\nabla\psi_2}^2-2\nabla\psi_2\cdot A_2-{\rm i}div(A_2)-\abs{A_2}^2\right)b_2$.
Combining this with \eqref{l1i}-\eqref{l1j} and repeating our previous arguments we obtain \eqref{l1b}.\qed

\subsection{Asymptotic properties of $S_1-S_2$ and representation formulas for $A_1-A_2$ and $V_1-V_2$}

In this subsection we will apply formulas \eqref{l1a}-\eqref{l1b} in order to derive two expressions involving $A_1-A_2$ and $V_1-V_2$ from   the asymptotic expansion of $S_1-S_2$,  as $\im\lambda\to+\infty$. For this purpose, we start by specifying our choice for the parameter $\lambda$, the function $A_{j,\sharp}$, the vector $\eta_j$, $j=1,2$,  and the function $b_2$ appearing in \eqref{anz}. Let us first define  the parameter $\lambda$ and the vectors $\eta_1$, $\eta_2$. We consider an arbitrary $\xi \in \R^n \setminus \{0\}$ and pick $\eta \in \mathbb S^{n-1}$ such that $\eta \cdot \xi=0$. Then, for $\tau>\abs{\xi}$ we put 
\bel{BA1} 
B_\tau=\sqrt{1-\frac{\abs{\xi}^2}{4\tau^2}},\ \eta_1(\tau)=B_\tau\eta-\frac{\xi}{2\tau},\ \eta_2(\tau)=B_\tau\eta+\frac{\xi}{2\tau}\ \mbox{and}\ \lambda(\tau)=(\tau+i)^2,
\ee
in such a way that 
\bel{BA2}\left\{\begin{aligned} \eta_1,\eta_2\in\mathbb S^{n-1},\\
\sqrt{\lambda}(\eta_1-\eta_2)\to-\xi,\quad \textrm{as } \tau\to+\infty,\\
\im\lambda\to+\infty,\quad \textrm{as } \tau\to+\infty,\\
 \im \sqrt{\lambda}\eta_1,\im \sqrt{\lambda}\eta_2\ \ \textrm{are bounded with respect to } \tau>\abs{\xi}.\end{aligned}\right.\ee
In order to get a suitable expression of the functions  $A_{j,\sharp}$, we first need to extend identically the magnetic potentials $A_j$, $j=1,2$. For this purpose we set $\tilde{\Omega}$  an arbitrary open bounded set of $\R^n$ such that $\overline{\Omega}\subset\tilde{\Omega}$ and we define $\tilde{A_1}\in \mathcal C^1_0(\tilde{\Omega},\R^n)$ such that ${\tilde{A_1}}_{|\Omega}=A_1$. Then,  we define $\tilde{A_2}$ by
\[\tilde{A}_2(x)=\left\{\begin{array}{l} A_2(x),\ \textrm{for }x\in\Omega,\\ \tilde{A}_1(x),\ \textrm{for }x\in\tilde{\Omega}\setminus\Omega.\end{array}\right.\]
In view of \eqref{t1a}, it is clear that $\tilde{A}_2\in \mathcal C^1_0(\tilde{\Omega},\R^n)$.  
We define the functions  $A_{j,\sharp}\in\mathcal C^\infty_0(\R^n;\R^n)$, $j=1,2$, by
\[A_{j,\sharp}(x):=\chi_{\delta}*\tilde{A}_j(x)=\int_{\R^n}\chi_\delta(x-y)\tilde{A}_j(y)dy,\]
where $\chi_\delta(x)=\delta^{-n}\chi(\delta^{-1}x)$, with $\delta>0$, is the usual mollifer with $\chi\in\mathcal C_0^\infty(\R^n)$, supp$(\chi)\subset\{x\in\R^n:\ |x|\leq 1\}$, $\chi\geq 0$ and $\int_{\R^n}\chi dx=1$. From now on we set $\delta=\tau^{-{1\over 3}}$ and we recall that
\[\psi_j(x)=-\int_{-\infty}^{0}\eta_j\cdot A_{j,\sharp}(x+s\eta_j)ds.\]
 We set also
\bel{b2}b_2(x)= e^{{\rm i}\omega\cdot x}y\cdot\nabla \left[ \exp\left( -{\rm i}\int_\R \eta_2\cdot A_\sharp(x+s\eta_2)ds\right)e^{-{\rm i}\omega\cdot x}\right],\ee
where $A_\sharp=A_{2,\sharp}-A_{1,\sharp}$, $\omega=B_\tau\xi-{|\xi|^2\eta\over 2\tau}\in\eta_2^\bot$, $B_\tau=\sqrt{1-{|\xi|^2\over 4\tau^2}}$, and 
\[b(x)= e^{{\rm i}x\cdot\xi}y\cdot\nabla \left[ \exp\left( -{\rm i}\int_\R \eta\cdot A(x+s\eta)ds\right)e^{-{\rm i}x\cdot\xi}\right],\quad \psi(x)=\int_{-\infty}^{0}\eta\cdot A(x+s\eta)ds.\]
Here $y\in\mathbb S^{n-1}\cap\eta^{\bot}$, $y\cdot\nabla$ denotes the derivative in the $y=(y_1,\ldots,y_n)$ direction given by 
$$y\cdot\nabla=\sum_{j=1}^ny_j\partial_{x_j}$$
and $A$ is the function defined by $A_2-A_1$ on $\Omega$  extended by $0$ outside of $\Omega$. Note that, in view of condition \eqref{t1a} we have $A\in \mathcal C^1_0(\Omega)$.
Since  $\tilde{A}_j\in \mathcal C^1_0(\R^n,\R^n)$, we find 
\bel{mol1}\norm{A_{j,\sharp}-A_j}_{L^\infty(\Omega)}\leq \norm{A_{j,\sharp}-\tilde{A}_j}_{L^\infty(\R^n)}\leq C \delta=C \tau^{-{1\over 3}}\ee
with $C$ depending on  $\Omega$ and any $ M\geq \underset{j=1,2}{\max}\norm{\tilde{A_j}}_{W^{1,\infty}(\R^n)}$. On the other hand, one can check that 
\bel{mol2}\norm{\partial_x^\alpha A_{j,\sharp}}_{L^\infty(\R^m)}\leq C\delta^{\abs{\alpha}-1}=C\tau^{{\abs{\alpha}-1\over 3}},\quad \alpha\in\mathbb N^n\setminus\{0\},\ee
where $C$ depends on $\Omega$ and any $ M\geq \underset{j=1,2}{\max}\norm{\tilde{A_j}}_{W^{1,\infty}(\R^n)}$. 

\begin{remark} Let us observe that, our anstazs are related  to the principal part of the complex geometric optics solutions of  \cite{SUU} and the extension of this construction to magnetic Schr\"odinger operators by \cite{FKSU,KLU,KU,NSU1,Sa1,Su}. Nevertheless, in contrast to the complex geometric optics solutions of \cite{SUU}, the large parameter of the ansatzs \eqref{anz}, that will be send to $+\infty$ for the uniqueness result, is given by $\im\lambda$ where the parameter $\lambda$ appears explicitly in \eqref{eq:lambda}. This makes it possible to construct ansatzs bounded with respect to the large parameter and to use the resolvent $(H_j-\lambda)^{-1}$, $j=1,2$, for the construction of a remainder term that admits a decay with respect to the large parameter $\im\lambda$. Moreover, in contrast to the geometric optics solutions of \cite{SUU}, whose principal parts take the form  \eqref{anz} when $\zeta_j\cdot\zeta_j=0$, our construction is not restricted to dimension $n\geq3$. Indeed,  the vector  $\zeta_j$, $j=1,2$, that we consider in the present paper are subjected  only to the condition \eqref{BA2}, already considered by \cite{I},  which requires only the two orthogonal vectors $\eta$ and $\xi$ appearing in \eqref{BA1}. For this reason, in contrast to the construction of \cite{SUU}, that requires three orthogonal vectors, our construction works also for $n=2$.\end{remark}

From now on, our goal is to derive  from \eqref{l1a}-\eqref{l1b} two formulas from some asymptotic properties of $S_1-S_2$ as $\tau\to+\infty$. For this purpose we need the following  intermediate result which follows from \eqref{mol1} and \eqref{mol2}.

\begin{lemma}\label{ll2} 
Let  the condition introduced above be fulfilled. Then, we have
\bel{ll2a} \sup_{\tau>|\xi|+1}\norm{b_2}_{L^\infty(\R^n)}<\infty\ee
and  
\bel{ll2b}\lim_{\tau\to+\infty} b_2(x)=b(x),\quad \lim_{\tau\to+\infty}\psi_1(x)-\psi_2(x)=\psi(x),\ x\in\R^n.\ee \end{lemma}
\begin{proof} Note first that
\bel{ll2c}b_2(x)=\left(-{\rm i}\omega\cdot y-{\rm i}\int_\R \eta_2\cdot y\cdot\nabla A_\sharp(x+s\eta_2)ds\right)\exp\left( -{\rm i}\int_\R \eta_2\cdot A_\sharp(x+s\eta_2)ds\right).\ee
On the other hand, we have $|\omega|\leq 1+|\xi|$ and, since $\tilde{A}_2-\tilde{A}_1$ is compactly supported and $\tilde{A}_2-\tilde{A}_1\in \mathcal C^1_0(\R^n,\R^n)$, we find $y\cdot\nabla A_\sharp=\chi_\delta*\left(y\cdot\nabla  (\tilde{A}_2-\tilde{A}_1)\right)$. Therefore, we obtain
\[\norm{b_2}_{L^\infty(\R^n)}\leq 1+|\xi|+C\norm{\chi_\delta}_{L^1(\R^n)}\norm{y\cdot\nabla  (\tilde{A}_2-\tilde{A}_1)}_{L^\infty(\R^n,\R^n)}\leq 1+|\xi|+CM\]
with $C$ a generic constant  depending only on $\Omega$  and $ M\geq \underset{j=1,2}{\max}\norm{\tilde{A_j}}_{W^{1,\infty}(\R^n)}$. From this last estimate we deduce \eqref{ll2a}. Now let us prove \eqref{ll2b}. Since $\tilde{A}_1$ and $\tilde{A}_2$ coincide outside of $\Omega$, we have $\tilde{A}_2-\tilde{A}_1=A$. Therefore, we deduce  that $A_\sharp=\chi_\delta* A$ and
\bel{ll2d}\abs{y\cdot\nabla A_\sharp(x+s\eta_2)-y\cdot\nabla A(x+s\eta)}\leq \abs{y\cdot\nabla A_\sharp(x+s\eta_2)-y\cdot\nabla A_\sharp(x+s\eta)}+\abs{y\cdot\nabla A_\sharp(x+s\eta)-y\cdot\nabla A(x+s\eta)}.\ee
The second term on the right hand side of this estimate can be rewritten as 
\[y\cdot\nabla A_\sharp(x+s\eta)-y\cdot\nabla A(x+s\eta)=\chi_\delta*[y\cdot\nabla A](x+s\eta)-y\cdot\nabla A(x+s\eta)\]
and since $A\in C^1_0(\R^n)$, we get
\bel{ll2e} \lim_{\tau\to+\infty} y\cdot\nabla A_\sharp(x+s\eta)-y\cdot\nabla A(x+s\eta)=0,\quad x\in\R^n,\ s\in\R.\ee
For the first term on the right hand side of \eqref{ll2d}, using the fact that for $\tau$ sufficiently large we have
\[\eta_2=\eta+{\xi\over 2\tau}+ \underset{\tau\to+\infty}{ o}\left({1\over\tau}\right)\]
and applying \eqref{mol2}, we get
\[\abs{y\cdot\nabla A_\sharp(x+s\eta_2)-y\cdot\nabla A_\sharp(x+s\eta)}\leq \norm{A_\sharp}_{W^{2,\infty}(\R^n)}\abs{s(\eta-\eta_1)}\leq C\abs{s}\tau^{-{2\over3}}\]
with $C$ depending on $\xi$, $\Omega$, $\tilde{A_1}$ and $\tilde{A_2}$. In view of this estimate we have
\[\lim_{\tau\to+\infty} y\cdot\nabla A_\sharp(x+s\eta_2)-y\cdot\nabla A_\sharp(x+s\eta)=0,\quad x\in\R^n,\ s\in\R.\]
Combining this last result with \eqref{ll2d}-\eqref{ll2e}, we get
\[\lim_{\tau\to+\infty}y\cdot\nabla A_\sharp(x+s\eta)=y\cdot\nabla A(x+s\eta),\quad x\in\R^n,\ s\in\R.\]
Then, using the fact that supp$(A_\sharp)\subset \Omega+\{x\in\R^n:\ |x|\leq \delta\}$ and \eqref{mol2}, by the dominate convergence theorem we get that 
\[\lim_{\tau\to+\infty}\int_\R y\cdot\nabla A_\sharp(x+s\eta_2)ds=\int_\R y\cdot\nabla A(x+s\eta)ds,\quad \textrm{$x\in\R^n$}.\]
Putting this together with \eqref{ll2c} and the fact that $\omega\to\xi$, $\eta_2\to\eta$ as $\tau\to+\infty$,  we obtain
\[\lim_{\tau\to+\infty}b_2(x)=\left(-{\rm i}\xi\cdot y+-{\rm i}\int_\R \eta\cdot y\cdot\nabla A(x+s\eta)ds\right)\exp\left( -{\rm i}\int_\R \eta\cdot A(x+s\eta)ds\right)=b(x),\ x\in\R^n.\]
Using similar arguments we deduce that
\[\lim_{\tau\to+\infty}\psi_1(x)-\psi_2(x)=\psi(x)=\int_{-\infty}^{ 0} \eta\cdot A(x+s\eta)ds,\quad x\in\R^n.\]
This completes the proof of the lemma.\end{proof}

Applying \eqref{l1a}-\eqref{l1b}, \eqref{mol1}-\eqref{ll2b} and sending $\tau\to+\infty$, we obtain our first formula involving the magnetic potentials $A_1, A_2$.

\begin{proposition}
\label{l2} 
 Fix $\xi\in\R^n\setminus \{0\}$  and $\eta\in \mathbb S^{n-1}$ such that $\eta\cdot\xi=0$.  Let  $\lambda$, $\eta_1$ and $\eta_2$ be defined by \eqref{BA1}  and let $b_2$ be defined by \eqref{b2}. Then,  we have
\bel{l2a}
\lim_{\tau \to+\infty} \frac{S_1-S_2}{\sqrt{\lambda}}=2\int_\Omega \eta\cdot(A_1-A_2)e^{-{\rm i}\xi\cdot x} b e^{{\rm i}\psi(x)}dx. 
\ee
\end{proposition}
\begin{proof} 
With reference to \eqref{jt1} and \eqref{BA1} we have
$\abs{\Phi_{1}(x)}=e^{-\eta_1 \cdot x}$ and $\abs{e^{-{\rm i}\sqrt{\lambda} \eta_2 \cdot x}}=e^{\eta_2 \cdot x}$ for all $x\in\Omega$, 
hence $\norm{\Phi_{1}}_{L^2(\Omega)}^2 =  \int_{\Omega} e^{-2 \eta_1 \cdot x} dx \leq  | \Omega | e^{2 | \Omega |}$ and
$\norm{e^{-{\rm i}\sqrt{\lambda} \eta_2 \cdot x}}_{L^2(\Omega)}^2 \leq  | \Omega | e^{2 | \Omega |}
$ since $|\eta_1|=|\eta_2|=1$.  Moreover, in view of \eqref{BA1}, we have  the estimate
$$\norm{(H_j-\lambda)^{-1}}_{\mathcal B(L^2(\Omega))} ={1 \over \textrm{dist}(\lambda,\sigma(H_j)) } \leq {1 \over | \im \lambda |}= {1\over 2\tau},\quad j=1,2.$$
In addition, in light of \eqref{mol2}, we get
\[\norm{\psi_j}_{W^{2,\infty}(\Omega)}\leq C\delta=C\tau^{{1\over3}},\quad \norm{b_j}_{W^{2,\infty}(\Omega)}\leq C\delta^2=C\tau^{{2\over3}}\]
with $C$ a generic constant depending on $\xi$, $\Omega$ and $\tilde{A_j}$, $j=1,2$. 
Putting these estimates together with \eqref{t1a}, \eqref{l1a}-\eqref{l1b} and \eqref{mol1} , we deduce that
\[ \frac{S_1-S_2}{\sqrt{\lambda}}=2\int_\Omega \eta_2\cdot (A_1-A_2)e^{{\rm i}\sqrt{\lambda}({\eta}_1-{\eta}_2)\cdot x} b_2 e^{{\rm i}(\psi_1(x)-\psi_2(x))}dx+\underset{\tau\to+\infty}{\mathcal O}\left(\tau^{-{1\over3}}\right).\]
Combining this with \eqref{BA2}, \eqref{ll2a}-\eqref{ll2b} and applying the dominate convergence theorem we deduce \eqref{l2a}.

\end{proof}
Using similar arguments and assuming that the magnetic potentials are known ($A_1=A_2$), we obtain our second  formula involving the electric potentials $V_1,V_2$.
\begin{proposition}
\label{l3} 
 Assume that $A_1=A_2$. Fix $\xi\in\R^n\setminus \{0\}$  and $\eta\in \mathbb S^{n-1}$ such that $\eta\cdot\xi=0$. Let  $\lambda$, $\eta_1$ and $\eta_2$ be defined by \eqref{BA1} and $b_2=1$. Then,  we have
\bel{l3a}
\lim_{\tau \to+\infty} S_1-S_2=\int_\Omega (V_1-V_2)e^{-{\rm i}\xi\cdot x}  dx. 
\ee
\end{proposition}
\begin{proof} Note that for $A_1=A_2$ we have $q_{11}-V_1=q_{21}-V_2$,  $q_{12}=q_{22}$, $A_{1,\sharp}=A_{2,\sharp}$. Therefore, we deduce that \eqref{l1a}-\eqref{l1b} imply
\bel{l3b}\begin{aligned}S_1-S_2=&\int_\Omega (V_1-V_2)e^{{\rm i}\sqrt{\lambda}({\eta}_1-{\eta}_2)\cdot x}  e^{{\rm i}(\psi_1(x)-\psi_2(x))}dx-\int_\Omega \left[\lambda\left((H_1-\lambda)^{-1}-(H_2-\lambda)^{-1}\right)Q_1\right]Q_2dx\\
\ & -\int_\Omega \left[\sqrt{\lambda}(H_1-\lambda)^{-1}Q_1\right]V_1e^{-{\rm i}\sqrt{\lambda}{\eta}_2 \cdot x}e^{-{\rm i}\psi_2}dx-\int_\Omega \left[\sqrt{\lambda}(H_1-\lambda)^{-1}V_1\Phi_1\right]Q_2dx\\
\ &-\int_\Omega \left[(H_1-\lambda)^{-1}V_1\Phi_1\right]V_1e^{-{\rm i}\sqrt{\lambda}{\eta}_2 \cdot x}e^{-{\rm i}\psi_2}dx
+\int_\Omega \left[\sqrt{\lambda}(H_2-\lambda)^{-1}Q_1\right]V_2e^{-{\rm i}\sqrt{\lambda}{\eta}_2 \cdot x}e^{-{\rm i}\psi_2}dx\\
\ &+\int_\Omega \left[\sqrt{\lambda}(H_2-\lambda)^{-1}V_2\Phi_1\right]Q_2dx+\int_\Omega \left[(H_2-\lambda)^{-1}V_2\Phi_1\right]V_2e^{-{\rm i}\sqrt{\lambda}{\eta}_2 \cdot x}e^{-{\rm i}\psi_2}dx,\end{aligned}\ee
where 
\[Q_1=2\eta_1\cdot (A_1-A_{1,\sharp})\Phi_1+{(q_{11}-V_1)\Phi_1\over\sqrt{\lambda}},\quad Q_2=\left(2\eta_2\cdot (A_1-A_{1,\sharp})-{q_{12}\over\sqrt{\lambda}}\right)e^{-{\rm i}\sqrt{\lambda}{\eta}_2 \cdot x}e^{-{\rm i}\psi_2}.\]
On the other hand, since $H_2-\lambda=H_1-\lambda-(V_1-V_2)$, for $\tau$ sufficiently large we have
	\[\begin{aligned}(H_1-\lambda)^{-1}-(H_2-\lambda)^{-1}&=(H_1-\lambda)^{-1}\left(\textrm{Id}-\left(\textrm{Id}-(V_1-V_2)(H_1-\lambda)^{-1}\right)^{-1}\right)\\
	\ &=-(H_1-\lambda)^{-1}\sum_{k=1}^\infty\left((V_1-V_2)(H_1-\lambda)^{-1}\right)^k.\end{aligned}\]
Combining this with the fact that $\im\lambda=2\tau$, $\abs{\lambda}\leq |\tau^2-1|+2\tau$, and the fact that
\[\norm{(H_1-\lambda)^{-1}}_{\mathcal B(L^2(\Omega))}+\norm{(V_1-V_2)(H_1-\lambda)^{-1}}_{\mathcal B(L^2(\Omega))}\leq {C\over|\im\lambda|}={C\over2\tau}\]
with $C$ depending only on  $V_1$, $V_2$ and $\Omega$, we deduce that
\bel{l3c}\sup_{\tau>|\xi|+1}\norm{\lambda\left((H_1-\lambda)^{-1}-(H_2-\lambda)^{-1}\right)}_{\mathcal B(L^2(\Omega))}<\infty.\ee
In addition, \eqref{mol1}-\eqref{mol2} imply 
\[\lim_{\tau\to+\infty}\norm{Q_1}_{L^\infty(\Omega)}=\lim_{\tau\to+\infty}\norm{Q_2}_{L^\infty(\Omega)}=0.\]
Putting this result together  with \eqref{BA2}, \eqref{l3b}-\eqref{l3c}, we obtain 
\[\limsup_{\tau \to+\infty} \abs{(S_1-S_2)-\int_\Omega (V_1-V_2)e^{{\rm i}\sqrt{\lambda}({\eta}_1-{\eta}_2)\cdot x}  e^{{\rm i}(\psi_1(x)-\psi_2(x))}dx}=0.\]
On the other hand, repeating the arguments of Lemma \ref{ll2}, we find
\[\lim_{\tau\to+\infty}\psi_1(x)-\psi_2(x)=\psi(x)=\int_{-\infty}^{0} \eta\cdot A(x+s\eta)ds=0\]
since $A_1=A_2$. Thus, applying the  dominate convergence theorem we obtain
\[\lim_{\tau\to+\infty}\int_\Omega (V_1-V_2)e^{{\rm i}\sqrt{\lambda}({\eta}_1-{\eta}_2)\cdot x}  e^{{\rm i}(\psi_1(x)-\psi_2(x))}dx=\int_\Omega (V_1-V_2)e^{-{\rm i}x\cdot\xi}dx\]
and we deduce \eqref{l3a}.\end{proof}

Armed with formulas \eqref{l2a}-\eqref{l3a}, in the next section we will complete the proof of Theorem \ref{thm-1}. 
\section{Proof of the main result}
\label{sec:Proof-main}
This section is devoted to the proof of our main result.  In all this section, for $j= 1$ and $j= 2$, we consider two magnetic potentials $A_j$ and electric potentials $V_{j}$ satisfying the assumptions of Theorem \ref{thm-1} and we denote by $H_j$ the associated operators defined by \eqref{eq:Def-A-theta} for $A=A_j$ and $V=V_j$. Let $(\lambda_{j,k},\phi_{j,k})_{k \geq 1}$ be a sequence of eigenvalues and eigenfunctions of $H_j$. In order to prove Theorem \ref{thm-1}, in light of \eqref{l2a}-\eqref{l3a}, we prove first that the condition
\bel{l4b}\lim_{\tau\to+\infty}{S_1(\lambda(\tau),\eta_1(\tau),\eta_2(\tau))-S_2(\lambda(\tau),\eta_1(\tau),\eta_2(\tau))\over\sqrt{\lambda(\tau)}}=0\ee
implies $dA_1=dA_2$. Then, we show that for $A_1=A_2$ the condition 
\bel{l4c}\lim_{\tau\to+\infty}S_1(\lambda(\tau),\eta_1(\tau),\eta_2(\tau))-S_2(\lambda(\tau),\eta_1(\tau),\eta_2(\tau))=0\ee
implies $V_1=V_2$. Finally, we complete the proof by proving that condition \eqref{t1a}-\eqref{tt2a} imply \eqref{l4b}-\eqref{l4c}.

We start by proving that \eqref{l4b} implies $dA_1=dA_2$.
\begin{lemma}
\label{l4} Let $\eta_1(\tau)$, $\eta_2(\tau)$ and $\lambda(\tau)$ be fixed by \eqref{BA1} and $b_2$ be defined by \eqref{b2}. Assume that \eqref{l4b} is fulfilled.
Then, we have $dA_1=dA_2$.\end{lemma}
\begin{proof}
Combining \eqref{l4b} with \eqref{l2a} we deduce that for all $\xi\in\R^n\setminus\{0\}$, $\eta\in\mathbb S^{n-1}$, satisfying $\eta\cdot\xi=0$, 
we get
\[\int_\Omega \eta\cdot(A_2-A_1)e^{-{\rm i}\xi\cdot x} b(x) e^{{\rm i}\psi(x)}dx=0.\]
Here $b$ takes the form
\[b(x)=e^{{\rm i}x\cdot\xi}y\cdot\nabla \left[\exp\left(-{\rm i}\int_\R\eta\cdot A(x+s\eta)ds\right)e^{-{\rm i}x\cdot\xi}\right]\]
with $y\in \mathbb S^{n-1}\cap\eta^\bot$. Then, applying  Fubini's theorem, we obtain
\[0=\int_{\R^n}\eta\cdot A(x)e^{-{\rm i}\xi\cdot x} b(x) e^{{\rm i}\psi(x)}dx=\int_{\eta^\bot}\int_\R \eta\cdot A(x'+t\eta)e^{{\rm i}\psi(x'+t\eta)}b(x')e^{-{\rm i}\xi\cdot x'}dtdx'.\]
Here we use the fact that $b(x)=b(x-(x\cdot\eta)\eta)$ and $\xi\cdot\eta=0$. On the other hand, for all $x'\in\eta^\bot$ and $t\in\R$, we have
\[\eta\cdot A(x'+t\eta)e^{{\rm i}\psi(x'+t\eta)}=\eta\cdot A(x'+t\eta)\exp\left({\rm i}\int_{-\infty}^t\eta\cdot A(x'+s\eta)ds\right)=-{\rm i}\partial_t \exp\left({\rm i}\int_{-\infty}^t\eta\cdot A(x'+s\eta)ds\right).\]
Therefore, we find
\[\begin{aligned}\int_{\R^n}\eta\cdot A(x)e^{-{\rm i}\xi\cdot x} b(x) e^{{\rm i}\psi(x)}dx&=-{\rm i}\int_{\eta^\bot}\left[\int_\R \partial_t \exp\left({\rm i}\int_{-\infty}^t\eta\cdot A(x'+s\eta)ds\right)dt \right]b(x')e^{-{\rm i}\xi\cdot x'}dx'\\
\ &=-{\rm i}\int_{\eta^\bot}  \left[\exp\left({\rm i}\int_\R\eta\cdot A(x'+s\eta)ds\right)-1\right] b(x')e^{-{\rm i}\xi\cdot x'}dx'.\end{aligned}\]
It follows
\bel{t1b}\int_{\eta^\bot}  \left[\exp\left({\rm i}\int_\R\eta\cdot A(x'+s\eta)ds\right)-1\right] b(x')e^{-{\rm i}\xi\cdot x'}dx'=0.\ee
We fix $i,j\in\{1,\ldots,n\}$ such that $i< j$ and we assume that $\xi \in\{\xi=(\xi_1,\ldots,\xi_n):\ \xi_i\neq0\}$. We can choose $\eta={\xi_j e_i-\xi_ie_j\over \sqrt{\xi_i^2+\xi_j^2}}$ and   $y={\xi_i e_i+\xi_je_j\over \sqrt{\xi_i^2+\xi_j^2}}\in\eta^\bot$. Here $(e_1,\ldots,e_n)$ is the canonical  basis of $\R^n$ defined by $e_1=(1,0,\ldots,0),\ldots, e_n=(0,\ldots,0,1)$.  Then, \eqref{t1b} implies
\[\int_{\eta^\bot}  \left[\exp\left({\rm i}\int_\R\eta\cdot A(x'+s\eta)ds\right)-1\right]y\cdot\nabla \left[\exp\left(-{\rm i}\int_\R\eta\cdot A(x'+s\eta)ds\right)e^{-ix'\cdot\xi}\right]dx'=0.\]
Integrating by parts we get
\[{-{\rm i}\over\sqrt{\xi_i^2+\xi_j^2}}\cdot\int_{\R^n}(\xi_jy\cdot\nabla a_i(x)-\xi_iy\cdot\nabla a_j(x))e^{-{\rm i}x\cdot\xi}dx=-{\rm i}\int_{\eta^\bot} \left( \int_\R\eta\cdot y\cdot\nabla A(x'+s\eta)ds\right)e^{-{\rm i}x'\cdot\xi}dx'=0\]
with $A=(a_1,\ldots,a_n)$. Integrating again by parts, we find
\[\begin{aligned}\int_{\R^n}(\xi_ja_i-\xi_ia_j)e^{-{\rm i}x\cdot\xi}dx&={y\cdot\xi\over\sqrt{\xi_i^2+\xi_j^2}}\int_{\R^n}(\xi_ja_i-\xi_ia_j)e^{-{\rm i}x\cdot\xi}dx\\
\ &={-{\rm i}\over\sqrt{\xi_i^2+\xi_j^2}}\cdot\int_{\R^n}(\xi_jy\cdot\nabla a_i(x)-\xi_iy\cdot\nabla a_j(x))e^{-{\rm i}x\cdot\xi}dx=0\end{aligned}\]
and it follows that  for all  $\xi \in\{\xi=(\xi_1,\ldots,\xi_n):\ \xi_i\neq0\}$ we have $\mathcal F[\partial_{x_j}a_i-\partial_{x_i}a_j](\xi)=0$.  On the other hand, since $\partial_{x_j}a_i-\partial_{x_i}a_j$ is compactly supported, $\mathcal F(\partial_{x_j}a_i-\partial_{x_i}a_j)(\xi)$ is continuous in $\xi\in\R^n$ and it follows $\mathcal F(\partial_{x_j}a_i-\partial_{x_i}a_j)=0$ on $\R^n$.  From this last result, we deduce that $\partial_{x_j}a_i-\partial_{x_i}a_j=0$  which implies that $dA_1=dA_2$. \end{proof}

Now assuming that $A_1=A_2$, we show in the next lemma that \eqref{l4c} implies $V_1=V_2$.
\begin{lemma}
\label{l66} Let $\eta_1(\tau)$, $\eta_2(\tau)$ and $\lambda(\tau)$ be fixed by \eqref{BA1} and $b_2=1$. Assume that $A_1=A_2$ and \eqref{l4c} is fulfilled. Then, we have $V_1=V_2$.\end{lemma}
\begin{proof} Fix $\xi\in\R^n\setminus\{0\}$ and choose $\eta\in\mathbb S^{n-1}\cap\xi^\bot$. Fix also $b=1$.
Thus, combining \eqref{l3a} and \eqref{l4c}, we find
\[\int_{\R^n} V(x)e^{-{\rm i}x\cdot\xi}dx=0\]
with $V=V_1-V_2$ extended by $0$ outside of $\Omega$. It follows that $V_1=V_2$.\end{proof}

According to Lemma \ref{l4}, \ref{l66},  the proof of Theorem \ref{thm-1} will be completed if we  show that conditions \eqref{tt2a}  imply conditions  \eqref{l4b}, \eqref{l4c}. For this purpose, we adapt the approach of \cite{KKS} to magnetic Schr\"odinger operators.  Let $f\in H^{1\over2}(\Gamma)$ being fixed, with the notations of Lemmas \ref{lem:Resolution} and \ref{lem:v-lambda-mu}, we denote by $v_{j,\lambda,\mu} := u_{j,\lambda} - u_{j,\mu}$ the solution of \eqref{eq:v-lambda} where $V$ is replaced by $V_{j}$ and $A$ by $A_j$. We fix also $h_{j,k}  := {\partial_\nu\phi_{j,k}}_{|\Gamma}$ $\alpha_{j,k} := \langle f,h_{j,k}\rangle$.
Recalling that in Lemma \ref{lem:z-mu} we have set $z_{\mu} = u_{1,\mu} - u_{2,\mu}$, in a similar way to \cite{KKS},  writing the above identity for $j=1$ and $j=2$, applying \eqref{t1a} and then subtracting the resulting equations, we end up with a new relation, namely
\begin{equation}\label{dif}
\begin{aligned}{(\partial_\nu +{\rm i}A_1\cdot\nu) u_{1,\lambda}}_{|\Gamma}-{(\partial_\nu +{\rm i}A_2\cdot\nu) u_{2,\lambda}}_{|\Gamma}&={\rm i}(A_1-A_2)\cdot\nu f+\partial_\nu u_{1,\lambda}  - 
\partial_\nu u_{2,\lambda} \\
\  &= 
\partial_\nu z_{\mu} 
+ \partial_\nu v_{1,\lambda,\mu} - \partial_\nu v_{2,\lambda,\mu} .\end{aligned}
\end{equation}
Now let us set 
\[ F_j(\lambda,\mu,f):={\partial_\nu  v_{j,\lambda,\mu}}_{|\Gamma},\quad j=1,2.\]
According to \eqref{eq:v-Normal-Deriv}, we have
\begin{equation}\label{eq:Def-F-m}
F(\lambda,\mu,f) := F_1(\lambda,\mu,f)-F_2(\lambda,\mu,f)=\sum_{k =1}^{+\infty} \left[{(\mu - \lambda)\alpha_{1,k}  \over (\lambda - \lambda_{1,k})(\mu - \lambda_{1,k})} \, h_{1,k}-{(\mu - \lambda)\alpha_{2,k}  \over (\lambda - \lambda_{2,k})(\mu - \lambda_{2,k})} \, h_{2,k}\right].
\end{equation}
Consider the following intermediate results.
\begin{lemma}\label{lem-6}
Let  $\eta_1,\eta_2,\lambda$ be given by \eqref{BA1}. Consider $\Phi_j$, $j=1,2$, with $\Phi_1$ introduced in the previous section and $\Phi_2=  e^{-{\rm i}\sqrt{\lambda} {\eta}_2 \cdot x} b_2e^{-{\rm i}\psi_2}$, where $b_2$ is defined by \eqref{b2} or $b_2=1$. Then, we have
\bel{ll6a}\sup_{\tau>1}\sum_{k=1}^\infty \abs{{\left\langle \Phi_1,h_{j,k}\right\rangle \over \lambda_{j,k}-\lambda }}^2<\infty,\quad  \sup_{\tau>1}\sum_{k=1}^\infty \abs{{\left\langle \overline{\Phi_2},h_{2,k}\right\rangle \over \lambda_{2,k}-\lambda }}^2<\infty,\ j=1,2.\ee\end{lemma}
\begin{proof}
We start with the first estimate of \eqref{ll6a} for $j=1$. According to Lemma \ref{lem:Resolution} the solution $u_{1,\lambda}$ of \eqref{eq1} for $f=\Phi_1$, $A=A_1$ and $V=V_1$, is given by
\[u_{1,\lambda}=\sum_{k=1}^\infty {\left\langle \Phi_1,h_{1,k}\right\rangle\over \lambda-\lambda_{1,k}}\phi_{1,k}.\]
Therefore, we have
\bel{ll6b}\norm{u_{1,\lambda}}_{L^2(\Omega)}^2=\sum_{k=1}^\infty \abs{{\left\langle \Phi_1,h_{1,k}\right\rangle \over \lambda_{1,k}-\lambda }}^2.\ee
On the other hand, in view of \eqref{l1d}, we have
\[\norm{u_{1,\lambda}}_{L^2(\Omega)}\leq \norm{\Phi_1}_{L^2(\Omega)}+\norm{\sqrt{\lambda}(H_1-\lambda)^{-1}\left[2\eta_1\cdot (A_1-A_{1,\sharp})+{q_{11}\over\sqrt{\lambda}}\right]}_{L^2(\Omega)}.\]
Here $q_{11}$ is the expression introduced in Lemma \ref{l1}.
Combining this with the fact that
\[\norm{\sqrt{\lambda}(H_1-\lambda)^{-1}}_{\mathcal B(L^2(\Omega))}\leq {|\tau+{\rm i}|\over |\im \lambda|}= {|\tau+{\rm i}|\over  2\tau}\leq 1\]
and the fact that, according to \eqref{mol1}-\eqref{mol2}, we have
\[\lim_{\tau\to+\infty}\norm{\eta_1\cdot (A_1-A_{1,\sharp})}_{L^\infty(\Omega)}=\lim_{\tau\to+\infty} \norm{{q_{11}\over\sqrt{\lambda}}}_{L^\infty(\Omega)}=0\]
 we deduce the first estimate of \eqref{ll6a} for $j=1$. In a same way, for $j=2$ using the fact that according to \eqref{mol2} we have
 \[(-{\rm i}\nabla +A_2)^2\Phi_1 + V_2\Phi_1 - \lambda \Phi_1=\underset{\tau\to+\infty}{\mathcal O}(\tau)\]
and repeating our previous arguments we deduce the first estimate \eqref{ll6a} for $j=2$. For the second estimate of \eqref{ll6a}, repeating the previous arguments we find
\[(-{\rm i}\nabla +A_2)^2\overline{\Phi_2}+ V_2\overline{\Phi_2} - \overline{\lambda}\ \overline{\Phi_2}=\overline{({\rm i}\nabla +A_2)^2\Phi_2+ V_2\Phi_2 -\lambda\Phi_2}=\underset{\tau\to+\infty}{\mathcal O}(\tau).\]
Combining this estimate with the fact that
 \[\abs{{\left\langle \overline{\Phi_2},h_{2,k}\right\rangle \over \lambda_{2,k}-\lambda }}= \abs{{\left\langle \overline{\Phi_2},h_{2,k}\right\rangle \over \lambda_{2,k}-\overline{\lambda }}}\]
since $\lambda_{2,k}\in\R$, we deduce the second estimate of \eqref{ll6a} by repeating the above arguments.\end{proof}

From now on we set 
$$\begin{aligned}G(\lambda,\mu,\Phi_1,\Phi_2)&:=\langle F(\lambda,\mu, \Phi_1), \overline{\Phi_2} \rangle\\
\ &=\sum_{k =1}^{+\infty} (\mu - \lambda)\left[{\left\langle \Phi_1,h_{1,k} \right\rangle \left\langle h_{1,k},\overline{\Phi_2}\right\rangle\over (\lambda - \lambda_{1,k})(\mu - \lambda_{1,k})} \, -{\left\langle \Phi_1,h_{2,k} \right\rangle \left\langle h_{2,k},\overline{\Phi_2}\right\rangle \over (\lambda - \lambda_{2,k})(\mu - \lambda_{2,k})} \right]. \end{aligned}$$
 Combining estimates \eqref{ll6a} with Lemma 4.3, 4.4, 4.5  of \cite{KKS}, we obtain the following.

\begin{lemma}\label{lem-0} Let the conditions of Theorem \ref{thm-1} be fulfilled and let $\eta_1,\eta_2,\lambda$ be given by \eqref{BA1}. Then, 
$G(\lambda,\mu,\Phi_1,\Phi_2)$ converge to $G_*(\lambda,\Phi_1,\Phi_2)$ as $\mu\to-\infty$ and $G_*(\lambda,\Phi_1,\Phi_2)$ converge to $0$ as $\tau\to +\infty$. Here we consider both the case $b_2$ given by \eqref{b2} and the case $b_2=1$.\end{lemma}

Armed with Lemma \ref{lem-0}, we are now in position to complete the proof of Theorem \ref{thm-1}. \\
\ \\
\textbf{Proof of Theorem \ref{thm-1}.} Note first that  according  \eqref{dif}, for $M=\norm{V_1}_{L^\infty(\Omega)}+\norm{V_2}_{L^\infty(\Omega)}$, we have
\[S_1(\lambda,\eta_1,\eta_2)-S_2(\lambda,\eta_1,\eta_2)=\left\langle \partial_\nu z_{\mu},e^{{\rm i}\overline{\sqrt{\lambda}} {\eta}_2 \cdot x} \overline{b_2}e^{{\rm i}\psi_2} \right\rangle+G(\lambda,\mu,\Phi_1,\Phi_2),\quad \mu\in(-\infty, -M),\]
where $\lambda$, $\eta_1$, $\eta_2$ are fixed by \eqref{BA1}, $b_2$ is given by \eqref{b2} or $b_2=1$  and $z_{\mu} = u_{1,\mu} - u_{2,\mu}$ with $u_{j,\mu}$, $j=1,2$,  the solution of \eqref{eq:v-lambda} where $\lambda$  is replaced by $\mu$, $V$ by $V_{j}$, $A$ by $A_j$ and $f$ by $\Phi_1$.
In view of Lemma \ref{lem:z-mu} and  Lemma \ref{lem-0}, sending $\mu\to-\infty$ we get
\[S_1(\lambda,\eta_1,\eta_2)-S_2(\lambda,\eta_1,\eta_2)= G_*(\lambda,\Phi_1,\Phi_2).\]
Then, in view of Lemma \ref{lem-0}, conditions \eqref{l4b} and \eqref{l4c} are fulfilled and according to Lemma \ref{l4} we have $dA_1=dA_2$. Therefore, condition \eqref{t1a} implies that for $A=A_2-A_1$ extended by $0$ outside of $\Omega$  we have $dA=0$ on $\R^n$. Thus, there exists $p\in W^{2,\infty}(\R^n)$ given by
$$p(x)=\int_0^1x\cdot A(tx)dt$$
such that $A=\nabla p$ on $\R^n$. Since  $\R^n\setminus \Omega$ is connected, applying the fact that $A=0$ on $\R^n\setminus \Omega$, upon eventually subtracting a constant we may assume that $p_{|\R^n\setminus \Omega}=0$ which implies that $p_{|\Gamma}=0$. Now let us consider the operator $H_3=(-{\rm i}\nabla+A_1)+V_2$ acting on $L^2(\Omega)$ with Dirichlet boundary condition and let $(\lambda_{3,k},\phi_{3,k})_{k \geq 1}$ be a sequence of eigenvalues and eigenfunctions of $H_3$. Since $A_1=A_2-\nabla p$ one can check that $H_3=e^{{\rm i}p}H_2e^{-{\rm i}p}$. From this identity we deduce that 
$$ \lambda_{3,k}=\lambda_{2,k},\quad k\geq1.$$
Moreover, for all $k\geq1$  we can choose $\phi_{3,k}=e^{ip}\phi_{2,k}$ and deduce that the condition
$$\partial_\nu \phi_{3,k}=\partial_\nu \phi_{2,k},\quad k\geq1$$
is also fulfilled.
Thus, conditions \eqref{tt2a} imply that
$$\lim_{k\to+\infty}|\lambda_{1,k}-\lambda_{3,k}|=0\quad\textrm{and}\quad \sum_{k=1}^{+\infty}\norm{\partial_\nu\phi_{1,k} -\partial_\nu\phi_{3,k}}_{L^2(\Gamma)}^2<\infty.$$
Then, repeating the arguments  of Lemma \ref{lem-0} we obtain
$$\lim_{\tau\to+\infty} \tilde{S}_1(\lambda(\tau), \eta_1(\tau),\eta_2(\tau))-\tilde{S}_3(\lambda(\tau), \eta_1(\tau),\eta_2(\tau))=0,$$
where 
$$\tilde{S}_j(\lambda,\eta_1,\eta_2)=\left \langle \Lambda_{j,\lambda} \Phi_1 , e^{{\rm i}\overline{\sqrt{\lambda}} {\eta}_2 \cdot x} e^{{\rm i}\tilde{\psi}_2}\right\rangle,\quad j=1,3$$
with 
$$\tilde{\psi}_2(x)=\int_{-\infty}^{ x\cdot\eta_2} \eta_2\cdot A_{1,\sharp}(x+(s-x\cdot\eta_2)\eta_2)ds,\quad b_2=1$$
and $\Lambda_{3,\lambda}$ the Dirichlet-to-Neumann map associated to problem \eqref{eq1} for $A=A_1$ and $V=V_2$. Then, in view of Lemma \ref{l66} we have $V_1=V_2$. This completes the proof of Theorem \ref{thm-1}.\qed

\section*{Acknowledgements}

 The author would like to thank the anonymous referee  for valuable comments and helpful remarks.  The author is grateful to  Otared Kavian and Eric Soccorsi   for their   remarks and  suggestions.

\end{document}